\documentclass[12pt,a4paper]{amsart}
          
\usepackage{mystyle}
\oddsidemargin \evensidemargin
\definecolor{bl1}{HTML}{1F558A}
\definecolor{pur1}{HTML}{52196D}
\definecolor{mag1}{HTML}{2AD0F1}
\title{Polar Degrees and Closest Points in Codimension Two}
\author{Martin Helmer and Bernt Ivar Utst$\o$l N$\o$dland}

\begin{document}

\begin{abstract} \noindent
Suppose that $X_A\subset \p^{n-1}$ is a toric variety of codimension two defined by an $(n-2)\times n$ integer matrix $A$, and let $B$ be a Gale dual of $A$. In this paper we compute the Euclidean distance degree and polar degrees of $X_A$ (along with other associated invariants) combinatorially working from the matrix $B$. Our approach allows for the consideration of examples that would be impractical using algebraic or geometric methods. It also yields considerably simpler computational formulas for these invariants, allowing much larger examples to be computed much more quickly than the analogous combinatorial methods using the matrix $A$ in the codimension two case. 

\end{abstract}
\maketitle
\section{Introduction}

To a projective variety $X \subset \p^n$ we may associate the \textit{polar varieties} of $X$; these are subvarieties of $X$ whose points have tangent spaces which intersect non-transversally with a fixed linear subspace. The classes of the polar varieties in the Chow ring are invariants of the projective embedding; in particular their degrees, which are often refereed to as \textit{polar degrees}, are projective invariants of $X$. As projective invariants, polar varieties and polar degrees have been historically important in the study and classification of projective varieties  \cite{aluffis2016paper,bank2010geometry,fulton,holme1988geometric,kleiman1986tangency,piene1978polar,PienePolar,tevelev2006projective}. In particular knowing the polar degrees of a smooth variety is equivalent to knowing the Chern classes of the tangent bundle, giving a simple expression for this Chern class. Polar varieties also arise in science and engineering problems where one tests the accuracy of mathematical models against observed data. In this setting it is natural to measure distance using the Euclidean norm and to compute the closest real point to some observed data within the model being studied. In the context of this Euclidean distance optimization problem the polar degrees can be used to compute the \textit{Euclidean distance degree}, a projective invariant which quantifies the difficulty of solving the optimization problem \cite{DHOST,HS,ottaviani2014exact}. 

In this paper we consider the situation where $X$ is a codimension two projective variety parameterized by monomials, i.e.,~$X$ is a codimension two projective toric variety. In this case we develop computationally simple formulas for the quantities which determine the polar degrees, Chern-Mather class, and Euclidean distance degree of a projective toric variety.  

We introduce the objects to be studied in this paper with an example from classical algebraic geometry which also arises in cell biology when studying pore forming cytotoxins used by numerous pathogenic bacteria \cite{AH17,los2013role}. Let $$
A=\begin{bmatrix}
3& 2&1&0\\
1& 1&1&1
\end{bmatrix},
$$this matrix gives rise to the \textit{twisted cubic curve} in $\p^3$ via the closure of the image of a monomial map defined by $A$, that is $$
X_A= \overline{\left\lbrace (t_1^3t_2:t_1^2t_2:t_1t_2:t_2)\; |\;(t_1,t_2)\in (\C^*)^2 \right\rbrace}\subset \p^3.
$$

The toric variety $X_A$ has codimension two. Let $k[x_0,x_1,x_2,x_3]$ be the coordinate ring of $\p^3$, the toric ideal of $X_A$ in this ring is the prime binomial ideal $$
I=({x}_{2}^{2}-{x}_{1} {x}_{3},{x}_{1} {x}_{2}-{x}_{0} {x}_{3},{x}_{1}^{2}-{x}_{0} {x}_{2}).
$$Consider the matrix $$
B=\begin{bmatrix}{-2}&
      {-1}\\
      3&
      1\\
      0&
      1\\
      {-1}&
      {-1}\\
      \end{bmatrix}.
$$The rows of the matrix $B$ generate the kernel of the linear map defined by the matrix $A$, we refer to $A$ and $B$ as Gale dual matrices and call $B$ the Gale dual of $A$.

Let $k[y_0,y_1,y_2,y_3]$ be the coordinate ring of $(\p^\vee)^3$. The conormal variety ${\rm Con}(X_A)$ of $X_A$ in $\p^3 \times (\p^\vee)^3$ parametrizes pairs of smooth points $x \in X_A$ and planes containing $T_{X_A,x}$. Its bigraded ideal is  defined by the sum of the ideal $I$ and the ideal 
defined by the $3\times 3$ minors ($3={\rm codim}(X_A)+1$) of the matrix$$
\begin{bmatrix}
0&
      {-{x}_{3}}&
      2 {x}_{2}&
      {-{x}_{1}}\\
      {-{x}_{3}}&
      {x}_{2}&
      {x}_{1}&
      {-{x}_{0}}\\
      {-{x}_{2}}&
      2 {x}_{1}&
      {-{x}_{0}}&
      0\\
y_0 & y_1 & y_2 & y_3 
\end{bmatrix}.
$$
The multidegree of the bigraded ideal defining ${\rm Con}(X_A)$ are the coefficients of the polynomial $4H^3h+3H^2h^2$, which represents the class $[{\rm Con}(X_A)]$ in the Chow ring $$A^*\left(\p^3 \times (\p^\vee)^3\right)\cong \Z[h,H]/(h^4,H^4).$$

Here $h$ denotes the class of a hyperplane in $\p^3$ and $H$ denotes the class of a hyperplane in $(\p^\vee)^3$. The polar degrees of $X_A$ are by definition this multidegree $$
(\delta_0(X_A),\delta_1(X_A))=(4,3).
$$ The first nonzero polar degree is the degree of the projective dual, in this case $\deg(X_A^{\vee})=\delta_0(X_A)=4$. 

From the polar degrees we may also determine the Chern-Mather class of $X_A$, $c_M(X_A)$ (since $X_A$ is smooth the Chern-Mather class agrees with the Chern class of the tangent bundle, i.e.~$c_M(X_A)=c(T_{X_A})\cap [X_A]$). The Chern-Mather class of $X_A$ (pushed forward to $A^*(\p^3)$) is $$
c_M(X_A)= 2h^3+3h^2\in A^*(\p^3)\cong \Z[h]/(h^4).
$$The Euclidean distance (ED) problem associated to $X_A$ seeks to determine the closest point in $X_A$ to a fixed generic point $u\in \R^4$. More specifically we wish to solve the optimization problem\small \begin{equation}
{\rm Minimize\; the\; function} \;f(t)=(t_1^3t_2-u_1)^2+(t_1^2t_2-u_2)^2+(t_1t_2-u_3)^2+(t_2-u_4)^2.\label{eq:firstEDex}
\end{equation}\normalsize
The critical points associated to this optimization problem are the solutions of the system of polynomial equations 
$$
\frac{\partial f}{\partial t_1}=\frac{\partial f}{\partial t_2}=0.
$$This polynomial systems will have $7$ non-zero complex solutions for generic data $u$, this number is the \textit{ED degree} of $X_A$. Observe that the ED degree of $X_A$ is equal to the sum of all non-zero polar degrees of $X_A$, this will be true in general (see \cite{DHOST}). In the context of systems biology solving this ED problem corresponds to testing if a particular model for pore-forming toxins describes experimentally measured data \cite{AH17}.

For a general projective variety the computation of the ED degree and polar degrees can become quite difficult as the degree of the generators and the dimension of the ambient space grows. This is true for all applicable algebraic or geometric methods (i.e.~Gr\"obner basis, homotopy continuation, etc.). For toric varieties, however, we may avoid the potentially time consuming algebraic or geometric methods and compute these invariants combinatorially. 

For a projective toric variety $X_A$ methods to compute the polar degrees, Chern-Mather class, and EDdegree based on the polyhedral combinatorics of the polytope $ {\rm Conv}(A)$ are given in \cite{HS}. It is also shown in \cite{HS} that much more computationally efficient formulas may be given in terms of the Gale dual of the matrix $A$ when ${\rm codim}(X_A)=1$, i.e.~when $X_A$ is a toric hypersurface. 

In this paper we develop analogous formulas for the polar degrees, Chern-Mather classes and ED degree of $X_A$ in terms of the Gale dual matrix, $B$, of $A$ for the next interesting case, when $X_A$ has codimension two. These formula yield substantially simpler expressions which are much faster to evaluate using a computer. 

The methods developed here build on the work of \cite{HS} and \cite{DS}. We also note that a method to compute the degree of the $A$-discriminant, i.e.,~the projective dual of $X_A$, in the codimension two case using the Gale dual matrix is given in \cite{DS}. Since this number must appear as the first non-zero polar degree our results also, in a sense, generalize their result.

Explicitly, by theorems of \cite{MT} and \cite{HS}, we know that the ED degree and polar degrees of a  projective toric variety $X_A$ are determined by the relative subdiagram volumes of the faces of the polytope $\Conv(A)$. Our contribution (detailed in \S\ref{section:MainResults}) is to give explicit formula for these subdiagram volumes in terms of the Gale dual matrix $B$ of $A$ when $X_A$ has codimension two. 

Working with the Gale dual is advantageous in low codimensions since in that case we work in a low dimensional integer lattice. This allows us compute the ED degree and polar degrees of large examples with a complicated face structure quickly. For example, in \S\ref{section:computation} we consider a projective toric variety $X_{A_6}$ of degree $581454473$ in $\p^9$ (the matrix $A_6$ is given in Appendix \ref{appendix:compEx}). Using the methods developed in this paper we compute ${\rm EDdegree}(X_{A_6})=74638158177$ in less than 30 seconds on a laptop, to find this number using algebraic or geometric methods would require computing the degree of a zero dimensional variety with over $74$ \textit{billion} isolated points in $\p^9$. Such a computation is unfeasible with current algebraic or geometric methods, even using a super computer. Using the combinatorial methods developed in \cite{HS} this computation takes over $2600$ seconds, hence our new combinatorial method gives a speed up of about 98 times in this case (see Table \ref{table:runtimes}). A Macaulay2 \cite{M2} package implementing the results developed in this paper can be found at the link \eqref{eq:website} below:
\begin{center}
\hfill \url{http://martin-helmer.com/Software/toricED_Codim2.html} \hfill  \eqnum\label{eq:website}
\end{center}

The paper is organized as follows, in \S\ref{section:background} we review background on computational tools and formulas that will be need in later sections. The main results are given in \S\ref{section:MainResults}. In \S\ref{section:computation} we test the performance of computer implementations on a variety of examples and analyze the theoretical computational complexity of our new combinatorial methods. 

\section{Background and Preliminaries}\label{section:background}
In this section we give background on toric varietes and their polar degrees, introduce Gale duality and gather some technical results needed in \S\ref{section:MainResults}.

\subsection{Toric Varieties, ED Degrees and Polar Degrees}

Let $A$ be a $d \times n$ integer matrix with columns $a_1,a_2,\ldots,a_n$, 
and rank $d$ such that the vector $(1,1,\ldots,1) $ lies in the row space of $A$ over $\Q$. Note that we allow $A$ to have negative entries.
Each column vector $a_i$ defines a monomial
$t^{a_i} = t_1^{a_{1i}} t_2^{a_{2i}} \cdots t_d^{a_{di}}$. The {\em affine toric variety} defined by $A$ is $$\,\tilde X_A\,=\,\overline{\{ (t^{a_1}, \ldots , t^{a_n}) \,:\,t \in (\C^*)^d \}}\subset \C^n,$$ that is, $\,\tilde X_A\,$ is the closure in $\C^n$ of the monomial parametrization specified by $A$. The affine toric variety $ \tilde{X_A}$ is the affine cone over the {\em projective toric variety} $X_A \subset \p^{n-1}$, that is $X_A$ is the closure in $\p^{n-1}$ of the image of the same monomial map. We have that ${\rm dim}(X_A) = d-1$ and ${\rm dim}(\tilde X_A) = d$. To the projective toric variety $X_A$ we can associate a polytope $P={\rm Conv}(A)$, which is the convex hull of the lattice points specified by the columns of the matrix $A$.  

We have a particular interest in the case when the toric variety $X_A$ has codimension two in $\p^{n-1}$; in this case $A$ is a $(n-2) \times n$ integer matrix. Let $B$ be a Gale dual matrix of $A$, i.e.~a $(n \times 2)$ matrix such that the image of $B$ equals the kernel of $A$. We will refer to a finitely generated free abelian group as a \textit{lattice}. Let $\Z B \subset \Z^n$ be the lattice spanned by the columns of the matrix $B$. We may associate a \textit{lattice ideal}, $I_B$, to the lattice $\Z B$ as follows:  \begin{equation}
I_B=\left( x^{l^+}-x^{l^-}\; | \; l \in \Z B \right) ,\label{eq:BinomialIdeal}
\end{equation}
 where $l^+_i$ is equal to $l_i$ if $l_i>0$ and $0$ otherwise, and where $l^-_i$ is equal to $|l_i|$ if $l_i<0$ and $0$ otherwise.
\begin{example} \label{runningexample}
The following example will be used throughout the paper to illustrate definitions and results. Consider the surface $X_A \subset \p^4$ arising from the matrix
\[ A = \begin{bmatrix} -2 & -2 & 1 & 0 & 0 \\ 4 & 0 & 0 & 1 & 0 \\ 1 & 1 & 1 & 1 & 1 \end{bmatrix}, {\rm \; with \;Gale\; dual}\;\;\; B=\begin{bmatrix} 1 & 0 \\ 0 & 1 \\ 2 & 2 \\ -4 & 0 \\ 1 & -3 \end{bmatrix}.\]
Explicitly the surface $X_A$ is given by \begin{equation}
X_A=\overline{\left\lbrace\left(\frac{t_2^4t_3}{t_1^2}:\frac{t_3}{t_1^2}:{t_3}{t_1}:{t_2}{t_1}:{t_3} \right)\; |\; t\in (\C^*)^3 \right\rbrace}=V(x_1x_3^2x_5-x_4^4,x_2x_3^2-x_5^3)\subset \p^4.
\end{equation} The polytope $P= \Conv(A)$ associated to $X_A$ is given in Figure \ref{figure:S2Ex}. $P$ has $3$ vertices $v_1,v_2,v_3$ and $3$ edges $e_1,e_2,e_3$.
\begin{figure}[h!]
  \begin{center} \begin{tikzpicture}[scale=0.9]
\draw [pur1,very thick](-2,4) -- (1,0);
\draw [pur1,very thick](1,0) -- (-2,0);
\draw [pur1,very thick](1,0) -- (-2,0);
\draw [pur1,very thick](-2,4) -- (-2,0);
\node at (-0.5,-.75) {${\color{pur1} e_1}$};
\node at (-2.75,2) {${\color{pur1} e_3}$};
\node at (0.15,2.35) {${\color{pur1} e_2}$};
\node at (1.65,-.35) {${\color{mag1} v_3}=a_3$};
\node at (-0.35,1) {$a_4$};
\node at (0,-0.35) {$a_5$};
\node at (-2.45,-0.35) {${\color{mag1}v_2}=a_2$};
\node at (-2.45,4.3) {${\color{mag1}v_1}=a_1$};
\fill[bl1] (0,1) circle[radius=2pt];
\fill[mag1] (-2,4) circle[radius=2pt];
\fill[mag1] (1,0) circle[radius=2pt];
\fill[bl1] (0,0) circle[radius=2pt];
\fill[mag1] (-2,0) circle[radius=2pt];
\end{tikzpicture}
\end{center}\caption{$P={\rm Conv}(A)$. \label{figure:S2Ex}}
\end{figure}
\end{example}

The rows of the matrix $B$ are denoted by $b_i$, $i=1,...,n$. We assume as in \cite[p.13]{DS} that all $b_i$ are non-zero. This assumption is equivalent to saying that $X_A$ is not a cone over a coordinate point. We also note that since $X_A$ is irreducible, the rows of $B$ necessarily generate $\Z^2$ \cite[p.9]{DS}. 

We now review the ED problem for toric varieties. For what follows we fix a vector $\lambda = (\lambda_1 ,\ldots,\lambda_n)$ of positive real numbers. We define the $\lambda$-weighted Euclidean norm on $\R^n$ to be
$|| x ||_\lambda = (\sum_{i=1}^n \lambda_i x_i^2)^{1/2}$.
For a given  $u \in \R^n$, we wish to compute a real point $v \in \tilde X_A$ which is closest to the given $u$.
In particular the Euclidean distance problem is the constrained optimization problem:
\begin{equation}
\label{eq:ED_GenForm}
 {\rm Minimize}\,\,\, || u-v||_\lambda \,\,\,
{\rm such\;  that} \,\,\, v \in \tilde X_A \cap \R^n.
\end{equation}
Alternatively, using the parametric description of $X_A$, we can formulate the ED problem as the unconstrained optimization problem
\begin{equation}
\label{eq:EDUnconst}
 {\rm Minimize}\,\,  \sum_{i=1}^n \lambda_i ( u_i - t^{a_i}  )^2 \,\,\,
\hbox{over all $\,\,t = (t_1,\ldots,t_d) \in \R^d$. } 
\end{equation}
For generic $u$ and $\lambda$ the number of complex critical points of \eqref{eq:ED_GenForm} is constant, we refer to this number as the \textit{Euclidean distance degree} of $X_A$ and denote it ${\rm EDdegree}(X_A) $. This matches the definition of ED degree given in \cite{HS,DHOST, ottaviani2014exact} for the toric variety $X_A$.
The ED degree quantifies the inherent algebraic complexity of finding and
 representing the exact solutions to the ED problems \eqref{eq:ED_GenForm} and \eqref{eq:EDUnconst}. For instance note that ${\rm EDdegree}(X_A) $ it is an upper bound for the number of local minima of the ED problem associated to $X_A$. Since the degree of the monomial map defining $X_A$ may be greater than one, i.e.~ we could have $[\Z^d: \Z A]>1$, the number of complex critical points of \eqref{eq:EDUnconst} is given by
 ${\rm EDdegree}(X_A) \cdot [\Z^d: \Z A]$.

The relation bewteen the ED degree and polar degrees is the following: The ED degree of a projective variety
$X \subset \p^{n-1}$ is equal to the sum of the polar degrees of $X$, see \cite[Theorem 5.4]{DHOST}, that is
 \begin{equation}
\label{eq:sumofpolar}
\mathrm{EDdegree}(X)\,\,=\,\, \, \delta_0(X)  + \delta_1(X) + \cdots + \delta_{n-1}(X).
\end{equation}
We now define the polar degrees of $X$ following the conventions of Fulton \cite{fulton},
Holme \cite{holme1988geometric}, Piene \cite{piene1978polar} and others. The $j$-th \textit{polar degree} of $X$, written $\delta_j(X)$, is the degree of the $j$-th \textit{polar variety} of $X$ with respect to a general linear subspace $\,\ell_j=\p^{j+{\rm codim}(X)}\subset \p^{n-1}$:
$$ P_j\,\,=\,\,\,\overline{\left\lbrace x \in X_{\mathrm{smooth}} \; | \; \dim( T_x X \cap \ell_j )\,\geq \,j+1 
\right\rbrace} \,\,\, \subset\,\,\, \p^{n-1}.$$
Following Kleiman \cite{kleiman1986tangency}, we can also obtain the polar degrees $\delta_j(X) $ from the rational equivalence class of the conormal variety in the {Chow ring} $A^*(\p^{n-1}\times (\p^{n-1})^{\vee})\cong \Z[H,h]/(H^n,h^n)$; in this convention $H$ denotes the rational equivalence class of a generic hyperplane from the $\p^{n-1}$ factor and $h$ denotes the rational equivalence class of a generic hyperplane from the $(\p^{n-1})^{\vee}$. The \textit{conormal variety} of $X$ is
 $$
{\rm Con}(X)=\overline{ \left\lbrace (p,L)\; | \; p \in X_{\rm reg} \;{\rm and}\; L \supseteq T_pX  \right\rbrace }\subset \p^{n-1}\times (\p^{n-1})^{\vee}.
$$ 
The ideal of ${\rm Con}(X)$ can be constructed as follows. Let $I_X$ be the ideal defining $X$ in the coordinate ring of $\p^{n-1}$ and let $\C[y]=\C[y_1,\dots, y_n]$ be the coordinate ring of $(\p^{n-1})^{\vee}$. Set $c={\rm codim}(X)$, and let $\mathcal{J}$ be the ideal defined by the $(c+1)\times (c+1)$-minors of the matrix $\left[ J(X)\; y \right]^{T}$, where $J(X)$ is the Jacobian of $X$. The ideal of ${\rm Con}(X)$ in $\C[x,y]$ is $\mathcal{K}=(I_X+\mathcal{J}):(I_{{\rm Sing}(X)})^{\infty}$. The Chow class of ${\rm Con}(X)$ is
$$
[{\rm Con}(X)]=\delta_0H^{n-1}h+\cdots+\delta_{n-2}Hh^{n-1}\in A^*(\p^{n-1}\times (\p^{n-1})^{\vee}),
$$where the integers $\delta_0=\delta_0(X),\dots,\delta_{n-2}=\delta_{n-2}(X)$ are the polar degrees of $X$ defined above. From the point of view of commutative algebra (i.e.,~in the terminology of Miller and Sturmfels \cite{MScca}) the polar degrees are the multidegree of the bigraded ideal $\mathcal{K}$.

\subsection{Polar Degrees and the Chern-Mather class of $X_A$ via ${\rm Conv}(A)$}\label{subsection:polarDegreesA}
The Chern-Mather class was first introduced by MacPherson in \cite{macpherson} and is a generalization of the total Chern class of the tangent bundle to singular varieties. In projective space we may express the Chern-Mather class in terms of the polar classes, and conversely may express the polar classes in terms of the Chern-Mather class \cite{piene1978polar},\cite{aluffis2016paper}; in the remainder of this paper we will employ the latter point of view. To this end we now review formulas for the polar degrees and ED degree of a projective toric variety $X_A$ in terms of the Chern-Mather class of $X_A$, $C_M(X_A)$. In the context of toric varieties, we see that the coefficients of this characteristic class $C_M(X_A)$ take the form of weighted normalized lattice volumes, which we will refer to as the \textit{Chern-Mather} volumes. The Chern-Mather volume will agree with the usual normalized volume when $X_A$ is smooth. 

The Chern-Mather volumes are defined in terms of the local Euler obstruction associated to a face of the polytope of a projective toric variety. The local Euler obstruction of a variety is a constructible function $\Eu : X \to \Z$. It was originally used by MacPherson to construct Chern classes for singular varieties \cite{macpherson}.

\begin{definition}
Given faces $\beta \subset \alpha$ we define $i(\alpha,\beta)$ as the index $[\Z \alpha \cap \R\beta : \Z \beta ]$, where $\R \beta$ is the linear subspace of $\R^d$ spanned by $\beta$. We also define the relative normalized subdiagram volume $\mu(\alpha,\beta)$ (cf.~\cite[Definition 3.8]{GKZ}) as follows: let $A_{\alpha}/\beta$ denote the image of the lattice points $A \cap \alpha$ in the quotient lattice $\Z \alpha / \Z \alpha \cap \R\beta$, then \begin{equation}
 \mu (\alpha,\beta) = \Vol(\Conv(A_\alpha/\beta)) - \Vol(\Conv(A_\alpha/\beta \setminus \{ 0 \}),\label{eq:subdiagramVol}
\end{equation}
where the volume is normalized with respect to the lattice $\Z \alpha / \Z \alpha \cap \R\beta$.
\end{definition}

\begin{proposition} \cite[Thm.~4.7]{MT}
For a projective toric variety  $X_A\subset \p^{n-1}$, the local Euler obstruction is constant on any torus-orbit and can be computed recursively as follows 
\begin{enumerate}
\item ${\rm Eu}(P)\,\,=\,\,1,$
\item $\displaystyle {\rm Eu}({\beta}) \quad =\sum_{{\alpha} {\rm \;s.t.}\; {\beta}{\;\rm is\; a} \atop                                                        
{\rm proper\; face \; of \;} {\alpha}}                                          
\!\!\!\!                                                                        
(-1)^{\dim({\alpha})-\dim({\beta})-1} { \mu(\alpha, \beta)} {i(\alpha,\beta)} {\rm Eu}(\alpha).$
\end{enumerate}\label{propn:EUPropDef}
\end{proposition}

\begin{example}
Consider the surface $X_A$ from Example \ref{runningexample}. We have that $i(P,e_3)=4$, while $i(P,e_1)=i(P,e_2)=1$. To compute the subdiagram volume $\mu$ using the matrix $A$ we compute normalized volumes. For instance we see that $\mu(P,v_1)=\Vol(P) - \Vol ( \Conv (a_2,a_3,a_4,a_5)) = 12-3=9$ and that $\mu(P,e_1)=1,\mu(P,e_2)=1, \mu(P,e_3)=2$.

Since this example is a surface all these numbers are easily computable from the definitions above using the $A$-matrix. However when $X_A$ has large dimension this approach becomes much harder. In \S\ref{section:MainResults} we develop formulas in terms of the $B$-matrix, which in this example recovers the above numbers, but has the advantage of working easily for any $X_A$ of codimension two (even when the dimension is very large).
\end{example}

\begin{definition}
For a face $\beta$ of $P$ we define $\Vol(\beta)$ as the volume of $\beta$, normalized with respect to the lattice  $\Z \beta$.
\end{definition}

\begin{definition}
Let $X_A\subset \p^{n-1}$ be a projective toric variety and let $P={\rm Conv}(A)$. The dimension $i$ \textit{Chern-Mather volume}, $V_i$, of $X_A$ is given by
\[ V_i = \sum_{{\beta \;{\rm a \; face \; of }P} \atop {{\rm with}\; \dim(\beta)=i}} \Vol(\beta)\Eu(\beta). \]
 When $X_A$ is smooth we have that $\Eu(\beta)=1$ for all faces of $P$, and $V_i$ is the sum of the normalized volumes of all dimension $i$ faces of $P$. \label{def:CM_Volume}
\end{definition}

Let $A^*(\p^{n-1})\cong \Z[h]/(h^n)$ denote the Chow ring of $\p^{n-1}$, with $h$ representing the rational equivalence class of a hyperplane in $\p^{n-1}$. We may express the push-forward of the Chern-Mather class of $X_A$ to $\p^{n-1}$ as \begin{equation}
c_M(X_A)=\sum_{i=0}^{\dim(X_A)} V_i h^{n-i}\in A^*(\p^{n-1}),
\end{equation}where the $V_i$ are the Chern-Mather volumes of Definition \ref{def:CM_Volume}. From \cite{PieneCM} we have that the Chern-Mather class, in the Chow ring of $X_A$, may be written as \begin{equation}
C_M(X_A)=\sum_{\alpha} {\rm Eu}(\alpha)[X_\alpha]\in A^*(X_A).
\end{equation} 

Using \cite[Thm.~1.2]{HS} we may also write the polar degrees of a projective toric variety $X_A\subset \p^{n-1}$ in terms of the Chern-Mather volumes of $X_A$ as 
\begin{equation}
\delta_i (X_A)= \sum_{j=i+1}^{n-2} (-1)^{n-3-j} \binom {j}{i+1}V_{j-1} 
\label{eq:polarDegressCM}
\end{equation} 
for $i=0,...,n-3$. Using the formula above and the fact that the ED degree is the sum of the polar degrees (see also \cite[Thm.~1.1]{HS}) we obtain:

\begin{equation}
{\rm EDdegree}(X_A)= \sum_{j=0}^{n-3} (-1)^{n-3-j}(2^{j+1}-1)V_{j} .
\label{eq:EDDegressCM}
\end{equation}

The main task, from a practical point of view, when computing the invariants discussed above is computing the expressions $\mu(\alpha,\beta)$ appearing in Proposition \ref{propn:EUPropDef}; giving formulas for this expression will be the main focus of \S\ref{section:MainResults}. In proving these results we will make use of the method of Helmer and Sturmfels \cite[Remark 2.2]{HS} stated here as Proposition \ref{propn:SubDIagVolume}. 

\begin{proposition} \label{proposition:HSAlgorithm}
Let $X_A\subset \p^{n-1}$ be a projective toric variety with associated polytope $P={\rm Conv}(A)$, let $\alpha$ be a face of $P$ and let $\beta$ be a face of $\alpha$. Order the columns of $A$ so that those in $\beta$ comes first, then those from $\alpha \setminus \beta$ and finally those in $A \setminus \alpha$. The row Hermite normal form of this reordered matrix has block structure
\[ \begin{bmatrix} * & * & * \\ 0 & C & * \\ 0 & 0 & * \end{bmatrix} \]
where the integer matrix $C$ has $\dim(\alpha)-\dim(\beta)$ rows and
\[ \mu (\alpha,\beta) = \Vol(\Conv(C \cup 0)) - \Vol(\Conv(C)).\]\label{propn:SubDIagVolume}
\end{proposition}

\subsection{Working with the Gale Dual}\label{subsection:workingWithGaleDual}

When $X_A$ is a codimension two projective toric variety the Gale dual matrix $B$ of $A$ has only two columns, meaning if we use $B$ we may work over integer lattices in $\Z^2$ rather than in the (often) much larger integer lattices in $\Z^{\dim(X_A)}$. This approach yields significant benefits in computational speed (see \S\ref{section:computation}) and also adds theoretical insights. In order to take advantage of this approach to compute polar degrees and other invariants we will need some basic results relating the structure of the Gale dual and the face structure of the polytope $P={\rm Conv}(A)$. 

\begin{proposition}\cite[Lemma 14.3.3]{CLS} \label{faces}
Fix $I \subset \{1,...,n\}$. The following are equivalent:
\begin{itemize}
\item There is a face $\beta$ of $P$ such that $a_i \in \beta$ if and only if $i \in I$.
\item There are positive numbers $t_i$ such that $\displaystyle\sum_{i \in I^c} t_ib_i = 0$.
\end{itemize}
\end{proposition}
Let $P={\rm Conv}(A)$ for a $d\times n $ integer matrix $A$. Motivated by Proposition \ref{faces} we define the following notations, for a face $\alpha$ of $P$ \begin{equation}
A_\alpha= \{ a_i\; | \;a_i \in A \cap \alpha\}, \;\; \; B_{\alpha}= \{b_i\; |\; a_i \notin A_\alpha \}. \label{eq:A_alpha__B_alpha}
\end{equation} Using a slight abuse of notation we let $A_\alpha$ (resp.~$B_\alpha$) denote both the sets above and also the matrices with columns $a_i$ (resp.~with rows $b_i$). We will also let $\mathfrak{I}_\alpha$ be the set of integer indices of the rows of $B_\alpha$.

In the case where $X_A\subset \p^{n-1}$ is a codimension two projective toric variety we can give a more specific description of the faces of $P$. In this case $A$ is an $(n-2) \times n$ integer matrix. For any proper face $\alpha$ of $P$ we have that $A_\alpha$ has either $\dim \alpha +1$ or $\dim \alpha + 2$ lattice points; otherwise we would contradict the assumption that all $b_i$ are non-zero. In the first of these cases $\alpha$ will be a simplex. Following \cite{DS} we make the following definition which, in terms of the $B$-matrix, singles out the faces of $P$ which are not simplices.

\begin{definition}
A line through $0$ in $\R^2$ is said to be a \textit{relevant line} if it contains two vectors $b_r,b_s$ in opposite directions.
\end{definition}

\begin{proposition} \label{proposition:faceDimension}
Let $X_A\subset \p^{n-1}$ be a projective toric variety and let $P={\rm Conv}(A) $ be the associated polytope. Let $\beta$ be a face of $P$. If all rows of $B_\beta$ are contained in the same relevant line then $\dim \beta = n - |B_\beta| -2$. If not then $\dim \beta = n - |B_\beta| -1$, in which case $\beta$ is a simplex.
\end{proposition}
\begin{proof}
Assume all $b_i$ are contained in a relevant line. Let $\beta=\beta_0 \subset \beta_1 \subset \cdots \subset \beta_r \subset P$ be a maximal chain of face inclusions. Since all $b_i$ are relevant we see that we can remove one $b_i$ from $\beta_0$ to get to the face $\beta_1$, remove two $b_i$ to get $\beta_2$ and so on. Thus $|B_{\beta_r}| = |B_\beta|-r$. By Proposition \ref{faces} a facet with points from a relevant line necessarily  has $2$ lattice points. Hence $|B_\beta|=r+2$ and $\dim \beta = (n-4)-r=n-4-(|B_\beta|-2)=n-|B_\beta|-2$.

Assume that not all $b_i$ are contained in the same relevant line. By a similar argument as above we can consider a maximal  inclusion of faces. Either the facet $\beta_r$ has $3$ lattice points, in which case $\dim \beta = n -|B_\beta| -1$, or it has two, which then has to be contained in a relevant line. However now there has to be an inclusion $\beta_i \subset \beta_{i+1}$ such that all lattice points of $\beta_{i+1}$ are in the relevant line, but not all in $\beta_i$. By Gale duality we must have that $|\beta_i| -|\beta_{i+1}| \geq 2$. From this it follows that $\dim \beta = n - |B_\beta| -1$. Since $\beta$ has $n-|B_\beta|$ lattice points it is a simplex.
\end{proof}

\begin{example}
For the matrix $A$ in Example \ref{runningexample} the corresponding Gale dual matrix $B$ is given by
\[ B=\begin{bmatrix} 1 & 0 \\ 0 & 1 \\ 2 & 2 \\ -4 & 0 \\ 1 & -3 \end{bmatrix},\]
note that $A\cdot B=0$. We see that the span of $(1,0)$ is a relevant line which corresponds to the fact that the edge $e_1$ in Figure \ref{figure:S2Ex} has three lattice points, instead of two.\label{runningexampleBdeg}
\end{example}

\begin{definition}
For a Gale dual matrix $B$ of an $(n-2)\times n$ integer matrix $A$, we define the notation $[i,j]:=\det(b_i,b_j)$, where $b_\ell$ denotes the $\ell^{th}$ row of $B$.
\end{definition} 
\begin{proposition} \label{proposition:dualMatrix}
Let $B$ be a given $2 \times n$ matrix such that the rows of $B$ span $\Z^2$ over $\Z$. Assume without loss of generality that $[1,2] \neq 0$. Then B is the Gale dual of the matrix
\[ A= \begin{bmatrix} [2,3] & -[1,3] & [1,2] & 0 & 0 & \cdots & 0 \\
[2,4] & -[1,4] & 0 & [1,2] & 0 &  \cdots & 0 \\
[2,5] & -[1,5] & 0 & 0 & [1,2] & \cdots & 0 \\
\vdots & \vdots & 0 & 0 & \ddots & \ddots  & 0 \\ 
[2,n-1] & -[1,n-1] & 0 & 0 &  \cdots & [1,2] & 0 \\
1 & 1 & 1 & 1 & & 1 \cdots & 1 \\\end{bmatrix} .\]
\end{proposition}
\begin{proof}
Writing out the matrix multiplication we see that $AB=0$, hence $\im(B) \subseteq \ker(A)$. Letting $v,w$ be generators of $\ker(A)$ we see that the columns $c_1,c_2$ of $B$ have to be of the form
\[ c_1 = pv+qw ,\;\;\; c_2 = sv+tw \;\;\; {\rm with}\;\;\; D=\begin{vmatrix} p & q \\ s & t \end{vmatrix} \neq 0.\]
A computation shows that every $2 \times 2$ minor of $B$ will have $D$ as a factor. Since the rows of $B$ span $\Z^2$ there must exist vectors $v,w$ in the rowspan of $B$ with $\det(v,w)=1$. Write $v = \sum_{i=1}^n a_ib_i, w=\sum_{i=1}^n c_ib_i$. Then $1 = \det(v,w) = \sum_{i,j=1}^n a_ic_j \det(b_i,b_j)$,  hence $|D|$ must be a factor of $1$ thus $|D| =1$ and the columns of $B$ form a basis of $\ker(A)$.
\end{proof}

\subsection{Computing Lattice Indices}\label{subsection:latticeIndex}
Lattice indices appear frequently in the main results presented in \S\ref{section:MainResults}. Our primary motivation in \S\ref{section:MainResults} is to provide effective formula to compute the invariants discussed in \S\ref{subsection:polarDegreesA}; hence we require explicit methods for lattice index computation. Consider a $d\times n$ integer matrix $A$ of full rank $d$. Let $\Z A$ denote the integer span of the columns of the matrix $A$. We wish to compute the lattice index $[\Z^d:\Z A]$. 
\begin{proposition}Let $A$ be an $d\times n$ integer matrix with full rank, ${\rm rank}(A)=d$. Also let $M_A$ denote the $d\times d$ integer matrix specified by the non-zero columns of the Hermite normal form of $A$ (computed by elementary integer column operations on $A$). We have that $$
[\Z^d:\Z A]=\det(M_A). 
$$\label{proposition:LatticeInd1}
\end{proposition}\begin{proof}
The matrix $A$ has rank $d$ (over $\Z$), this implies that the column space is spanned by $d$ vectors, and hence when we perform the elementary integer column operations to compute the Hermite normal form we will retain only $d$ non-zero columns. The matrix $M_A$ is then a square matrix whose entries are the coefficients of $\Z A$ in the standard basis for $\Z^d$, by \cite[Corollary~9.63]{AMA} the conclusion follows. 
\end{proof}
A second way to compute the lattice index $[\Z^d:\Z A]$ is given by the following proposition. 
\begin{proposition} \label{proposition:indexgcd}
Let $A$ be an $d \times n$ integer matrix with full rank, ${\rm rank}(A)=d$. Let $v={n \choose d}$ and let $c_1,\dots, c_v$ be the $v$ maximal (that is $d\times d$) minors of $A$. Then we have that 
\[ [\Z^d:\Z A]=\gcd(c_1,\dots, c_v). \]
\end{proposition}\begin{proof}
The $d \times d$ minors of $A$ generate what is called the $d^{th}$ Fitting ideal of $A$, $\rm{Fit}_d(A)$. By \cite[Lemma~11.2.1]{Keating} we have that the Fitting ideal is preserved by elementary (integer in our case) row or column operations on $A$, i.e.~we have that $\rm{Fit}_d(A)=\rm{Fit}_d(\rm{Herm}(A))$ where $\rm{Herm}(A)$ is the (column-wise) Hermite normal form of $A$. Since $\Z$ is a principal ideal domain $\rm{Fit}_d(A)$ is generated by $\gcd(c_1,\dots, c_v)$, and by Proposition \ref{proposition:LatticeInd1} we have that $\rm{Fit}_d(\rm{Herm}(A))$ is generated by $[\Z^d:\Z A]$; this gives the stated result. 
\end{proof}
We note that the second method to compute $[\Z^d:\Z A]$ is less computationally efficient, but could still be convenient in some cases. 

\begin{remark} \label{remark:gcdB}
Assume that $X_A\subset \p^{n-1}$ is a toric variety of codimension two and let $B$ be a Gale dual of $A$. Then, since the rows of $B$ span $\Z^2$ \cite[pg. 4]{DS}, we have by Proposition \ref{proposition:indexgcd} that $\gcd(\{[i,j]\}_{i,j}) = 1$.
\end{remark}
\begin{remark} \label{remark:gcdPrimitive}
We say that an integer vector $v$ in $\Z^l$ is a \textit{primitive vector} if $v$ is not a non-trivial integer multiple of another integer vector in $\Z^l$. Let $B$ be a $n\times 2$ integer matrix whose rows span $\Z^2$ and let $v$ be a primitive vector in $\Z^2$. It is well known that we can choose a basis for $\Z^2$ consisting of $v$ and another vector $w$, such that $\det(v,w)=1$. Using this basis we can write $b_i = a_i v + c_i w$ for some integers $a_i,c_i$. Then 
\[ \det(v,b_i) = c_i,\;\;\; {\rm and } \;\;\; \det(b_i,b_j) = a_ic_j-a_jc_i.\]
Note that if $\gcd(\det(v,b_i))= \gcd(c_i) > 1$ then $\gcd([i,j])>1$, which contradicts the assumption that the rows of $B$ span $\Z^2$ by Remark \ref{remark:gcdB}. Hence $\gcd(\det(v,b_i)) = 1$.
\end{remark}

\begin{proposition} \label{proposition:latticeIndexA}
Let $A$ be the $d\times n$ integer matrix from Proposition \ref{proposition:dualMatrix}. Then we have that $[\Z^d:\Z A]=[1,2]^{n-4}$.
\end{proposition}
\begin{proof}
Let $S$ denote the set of the last $n-4$ columns of the matrix $A$ from Proposition \ref{proposition:dualMatrix}; we will now apply Proposition \ref{proposition:indexgcd}. Each maximal minor of $A$ is the determinant of a $(n-2) \times (n-2)$ matrix $\mathfrak{m}$. This matrix $\mathfrak{m}$ will have at least $(n-4)$ columns coming from the set $S$, i.e.~$\mathfrak{m}$ will have at least $n-4$ columns with only two non-zero entries.

If $\mathfrak{m}$ has $(n-2)$ vectors from $S$ then $\mathfrak{m}$ is a lower triangular matrix and $\det(\mathfrak{m})=[1,2]^{n-3}$. If $\mathfrak{m}$ has $(n-3)$ vectors from $S$ then $\mathfrak{m}$ has one of the vectors $a_1$ or $a_2$ as a column. Performing determinant preserving row and column operations on $\mathfrak{m}$ yields a diagonal matrix, from this we obtain one of the following: $$\det(\mathfrak{m})=\pm[1,2]^{n-4}[2,i],\; \; {\rm or}\;\; \det(\mathfrak{m})=\pm[1,2]^{n-4}[1,i], \;{\rm for\;} i=3,...,n-1, \;\; {\rm or}$$ $$\det(\mathfrak{m})=\pm[1,2]^{n-4}\left( [1,2]-\sum_{i=3}^{n-1}[2,i] \right),\;\;{\rm or}\;\;\det(\mathfrak{m})=\pm[1,2]^{n-4}\sum_{i=2}^{n-1}[1,i] $$ with the choice depending on which of $a_1$ or $a_2$ appears in $\mathfrak{m}$ and on which column in $S$ does not appear in $\mathfrak{m}$.

Observe that since $\sum_{i=1}^n b_i = 0$ then, by elementary properties of determinants, we have that the two last cases can be rewritten as:
 $$\det(\mathfrak{m})=\pm[1,2]^{n-4}\left( [1,2]-\sum_{i=3}^{n-1}[2,i] \right) = \pm[1,2]^{n-4} (\sum_{i=1}^{n-1}[i,2]) = \pm[1,2]^{n-4}[2,n] , \;{\rm and}$$ 
$$\det(\mathfrak{m})=\pm[1,2]^{n-4}\sum_{i=2}^{n-1}[1,i] = \pm[1,2]^{n-4} [n,1] .$$

In the final case, if $\mathfrak{m}$ has $(n-4)$ vectors from $S$, then both $a_1,a_2$ appear as columns of $\mathfrak{m}$. Let $k_1,k_2$ be the indices of the columns vectors from $S$ which do not appear in $\mathfrak{m}$; note $k_1,k_2\in \left\lbrace 3,\dots,n \right\rbrace$. The current case has two subcases, first the situation where $k_1\neq k_2\neq n$ and second the situation where one of $k_1$ or $k_2$ is equal to $n$.  We may again perform elementary row and column operations to obtain a diagonal matrix. In the first subcase, where $k_1\neq k_2\neq n$, this computation gives: $$\det(\mathfrak{m})=\pm[1,2]^{n-5}(-[2,k_1][1,k_2]+[1,k_1][2,k_2]),\;k_1\neq k_2, \; k_1,k_2\in \left\lbrace3,4,\dots,n-1\right\rbrace.  $$
By the Pl\"{u}cker relation defining $G(2,4) \subset \p^5$ we have that $$-[2,k_1][1,k_2]+[1,k_1][2,k_2]=[1,2][k_1,k_2].$$ Hence $\det(\mathfrak{m})=[1,2]^{n-4}[k_1,k_2]$, where $3 \leq k_1 \neq k_2 \leq n-1$. Now consider the second subcase, that is the case where one of $k_1$ or $k_2$ is equal to $n$ (i.e~where the $n^{th}$ column of $A$ does not appear in $\mathfrak{m}$). Suppose (without loss of generality) that $k_2=n$, then we have that:
 $$
 \det(\mathfrak{m})=
 \pm[1,2]^{n-5}\left([1,k_1]\sum_{i=3,\;i\neq k_1}^{n-1}[2,i]-[2,k_1]\sum_{i=3,\;i\neq k_1}^{n-1}[1,i]-[1,2]([1,k_1]+[2,k_1])  \right)\;  .$$Again applying the Pl\"{u}cker relations as above we have that \begin{align*}
 \det(\mathfrak{m})=&
 \pm[1,2]^{n-5}\left([1,2]\sum_{i=3,\; i\neq k_1}^{n-1}[k_1,i]-[1,2]([1,k_1]+[2,k_1])  \right)\;  \\ 
 =& \pm[1,2]^{n-4}\left(\sum_{i=3,\; i\neq k_1}^{n-1}[k_1,i]-[1,k_1]-[2,k_1]  \right) \\
 =& \pm[1,2]^{n-4} \sum_{i=1}^{n-1}[k_1,i]  = \pm[1,2]^{n-4}[k_1,n]   .
 \end{align*}Putting all the cases together we see that the maximal minors of $A$ all have values $ \pm [1,2]^{n-4}[i,j]$ for some $i \neq j$. Hence their greatest common divisor equals $[1,2]^{n-4} \gcd([i,j])$. By Remark \ref{remark:gcdB} we have that $\gcd([i,j])=1$; the conclusion follows.
\end{proof}

\subsection{Other Results Needed for Gale Dual Computations}\label{subsection:miscResults}

In this subsection we collect some results on the degrees of lattice ideals, these results will be needed in \S\ref{section:MainResults}.

The degree of a codimension one lattice ideal is the degree of the defining polynomial.

\begin{proposition} \label{proposition:degreeCodim1}
The degree of a homogeneous lattice ideal $I_B$ associated to a $n \times 1$ matrix $B$ is given by
\[ \deg(I_B) =  \sum_{i | b_i > 0 } b_i .\]
\end{proposition}

If now $I_B$ is a codimension two lattice ideal we define the following: For each $i,j$, if $b_i$ and $b_j$ lie in the interior of opposite quadrants, then define 
\[ \nu_{ij}:= \min \{ |b_{i1}b_{j2}|, |b_{i2}b_{j1}| \}. \] Let $\beta_i$ be the sum of all positive entries in the $i^{th}$ column of $B$. 
\begin{proposition} \cite[Corollary 2.2]{DS} \label{proposition:degreeBmatrix}
The degree of a homogeneous lattice ideal $I_B$ associated to a $n \times 2$ integer matrix $B$ is given by
\[ \deg(I_B) = \beta_1 \beta_2 - \sum_{i,j} \nu_{ij} .\]
\end{proposition}

\begin{corollary}
If $I_B$ is a prime homogeneous lattice ideal associated to an integer matrix $B$ with Gale dual $A$ then 
\[ \Vol(\Conv(A)) = \beta_1 \beta_2 - \sum_{i,j} \nu_{ij} .\]
\end{corollary}

\begin{example}
For the $B$-matrix in Example \ref{runningexample}, all $\nu_{ij}$ equal zero, hence $\deg I_B = \beta_1 \beta_2 = 3 \cdot 4 = 12$.
\end{example}

\begin{definition}
For an inclusion $L \subset M$ of abelian groups, we define $T(M/L)$ to be the torsion subgroup and $|T(M/L)|$ to be its order.
\end{definition}

When proving our main results in \S\ref{section:MainResults} we will sometimes need to compute volumes of convex hulls of lattice points where some lattice points appear more than once. The following proposition shows that this is also expressible as a degree of a lattice ideal, hence the results above can be applied. 

\begin{theorem} \cite[Theorem 4.6]{degreeLatticeIdeals} \label{theorem:degreeBideal}
Given an integer matrix $B$ whose rows generate a $r$-dimensional lattice $\Z B \subset \Z^n$ and defining a codimension $r$ lattice ideal $I_B$, there exists a $(n-r) \times n$ matrix $A = [v_1,...,v_n]$ of rank $n-r$ such that $\Z B \subset \ker(A)$ and  
\[ \deg(I_B) = |T(\Z^n/\Z B)| \Vol(\Conv(0,v_1,...,v_n))\]
where the volume is normalized with respect to $\Z A$. 
\end{theorem}

\section{Computing invariants of codimension two toric varieties}
\label{section:MainResults}
Let $X_A$ be a codimension two projective toric variety and let $B$ be the Gale dual of the matrix $A$. In this case $A$ is an $(n-2) \times n$ integer matrix and $B$ is an $n\times 2$ integer matrix. From the results in \S\ref{subsection:polarDegreesA} we see that to compute the Chern-Mather volumes, polar degrees, and the ED degree of $X_A$, we must compute both the normalized relative subdiagram volumes of all chains of faces and the normalized volumes of all faces of the polytope $P={\rm Conv}(A)$. In this section we present our main results. These results give explicit closed form expressions for the required normalized volume and normalized subdiagram volume computations in terms of the Gale dual matrix $B$. Both a theoretical analysis and practical testing shows that the methods using the matrix $B$ offer a quite substantial computational performance gain relative to the methods of \cite{HS} when ${\rm codim}(X_A)=2$, see \S\ref{section:computation} for a discussion of this.
\begin{example}
As discussed above our goal in this paper is to compute the polar degrees and ED degree (and other associated invariants) using the Gale dual matrix $B$ of $A$ when $X_A$ has codimension two. Continuing Example \ref{runningexample} we now summarize the volumes and subdiagram volumes of the faces of the polytope $P={\rm Conv}(A)$ from Figure \ref{figure:S2Ex} in Table \ref{table:faces}. In this section we will develop the necessary results to fill in this table using only the matrix $B$.
\begin{table}[h!]
\centering
\resizebox{0.6\linewidth}{!}{
\begin{tabular}{@{} *6c @{}}
\toprule 
 \multicolumn{1}{c}     {\color{Ftitle} $\alpha$} &   {\color{Ftitle}$B_\alpha$} &   {\color{Ftitle}$\Vol(\alpha)$} & {\color{Ftitle}$\mu(P,\alpha)$} &   {\color{Ftitle} $i(P,\alpha)$}  &   {\color{Ftitle} $\Eu(\alpha)$} \\ 
 \midrule 
${\color{pur1}e_1}$ & $\{ b_1,b_4 \}$ & 3 & 1 & 1 & 1 \\
${\color{pur1}e_2}$ & $\{ b_2,b_4,b_5 \}$ & 1 & 1 & 1 & 1 \\
${\color{pur1}e_3}$ & $\{ b_3,b_4,b_5 \}$ & 1 & 2 & 4  & 8 \\
${\color{mag1}v_1}$ & $\{ b_2,b_3,b_4,b_5 \}$ & 1 & 9 & 1 & 0 \\
${\color{mag1}v_2}$ & $\{ b_1,b_3,b_4,b_5 \}$ & 1 & 8 & 1 & 2 \\
${\color{mag1}v_3}$ & $\{ b_1,b_2,b_4,b_5 \}$ & 1 & 2 & 1 & 0 \\
\bottomrule
 \end{tabular}}\vspace{1mm}
\caption{Invariants of $P$.\label{table:faces}}
 \end{table}
Using the information in Table \ref{table:faces} along with Definition \ref{def:CM_Volume} we obtain the Chern-Mather volumes
\[ V_0 = 0 + 2 + 0 = 2, \]
\[ V_1 = 3 \cdot 1 + 1 \cdot 1 + 1 \cdot 8 = 12, \]
\[ V_2 = 12. \]
Substituting these values into \eqref{eq:polarDegressCM}, and \eqref{eq:EDDegressCM} we have that the polar degrees and ED degree are:  
\[ \delta_0(X_A) = 3V_2-2V_1+V_0 = 14 ,\]
\[ \delta_1(X_A) = 3V_2-V_1 = 24, \]
\[ \delta_2(X_A) = V_2 = 12, \]
\[ \rm{EDdegree}(X_A) = \delta_0(X_A)+\delta_1(X_A)+\delta_2(X_A) =  50. \]
\end{example}
\subsection{Gale Dual Formulas for Subdiagram Volumes in Codimension Two}
As above we consider a codimension two toric variety $X_A$ in $\p^{n-1}$ and let $P={\rm Conv}(A)$. In this subsection we present several formulas for subdiagram volumes covering all possible expressions which could appear in the computation of the polar degrees and ED degree of $X_A$. Let $\alpha$ and $\beta$ be faces of $P$. These subdiagram volumes can be broadly grouped into two types, those of the form $\mu(P,\beta)$ and those of the form $\mu(\alpha,\beta)$ where $\beta\subset \alpha$.
\subsubsection{Subdiagram volumes $\mu(P,\beta)$}
\label{subsection:mu_A_beta}

Let $\beta$ be a face such that $B_\beta$ only has vectors from the same relevant line. Let $v$ be a primitive vector in the relevant line and  define $\lambda_i$ by $b_i=\lambda_i v$, for $b_i \in B_{\beta}$. With these notations we define
\[ v_+^\beta = \{ i \;| \;b_i \in B_\beta, \lambda_i > 0 \} \]
\[ v_-^\beta = \{ i \;| \;b_i \in B_\beta, \lambda_i < 0 \} .\]

\begin{theorem} \label{theorem:muPrelevant}
Let $X_A\subset \p^{n-1}$ be a projective toric variety of codimension $2$ and $P={\rm Conv} (A)$. Let $\beta$ be a face of $P$ having codimension $r$ with only lattice points from a relevant line with primitive vector $v$. Let the set $\mathfrak{I}_{\beta}$ index the rows of $B_{\beta}$, then
\[\mu(P,\beta) = \frac{\min\left( \sum_{i\in v_+^\beta} |\lambda_i|, \sum_{i\in v_-^\beta} |\lambda_i|\right)}{\gcd\left(\lambda_i\; | \; b_i=\lambda_iv\right)_{i\in \mathfrak{I}_{\beta}}},\; {\rm and}\;\;i(P,\beta) = \gcd \left(\lambda_i\; | \; b_i=\lambda_iv \right)_{i\in \mathfrak{I}_{\beta}}. \]
\end{theorem}
\begin{proof} 
We will apply Proposition \ref{propn:SubDIagVolume}. After reordering we may assume that $B_{\beta}$ consists of the rows $b_2,...,b_r$ of $B$. Since all lattice points of $B_{\beta}$ are contained in the same relevant line we have $[2,i]=0$ for all $i=3,...,r$. This implies that after reordering the columns as in Proposition \ref{propn:SubDIagVolume}   $A$ has block form
\begin{equation}
A= \begin{bmatrix}  D& \ast \\ 0 & C  \end{bmatrix} .\label{eq:AinmuPbetaProp}
\end{equation} 
 The submatrix $C$ has $r-2$ rows and is given by 
\[ C= \begin{bmatrix}  -[1,3] & [1,2] & 0 & 0 & \cdots & 0 \\
 -[1,4] & 0 & [1,2] & 0 &  \cdots & 0 \\
 -[1,5] & 0 & 0 & [1,2] & \cdots &  0 \\
\vdots & \vdots & \vdots & \vdots & \ddots & 0  \\ 
 -[1,r] & 0 & 0 &  \cdots & \cdots &  [1,2]  \\ \end{bmatrix}. \]
 By considering maximal minors we see that the lattice spanned by $C$ has index $j=[1,2]^{r-3} \gcd([1,i])_{i\in \mathfrak{I}_{\beta}}$. We know that $b_i=\lambda_iv$ for all $i=2,\dots, r$, let $|b_1,v|$ be the value of the determinant of the matrix with rows $b_1$ and $v$, then $ [1,i]=\lambda_i|b_1,v|$ and we have 
\begin{equation}
C= |b_1,v|\cdot \begin{bmatrix}  -\lambda_3 &\lambda_2 & 0 & 0 & \cdots & 0 \\
 -\lambda_4 & 0 & \lambda_2 & 0 &  \cdots & 0 \\
 -\lambda_5 & 0 & 0 & \lambda_2 & \cdots &  0 \\
\vdots & \vdots & \vdots & \vdots & \ddots & 0  \\ 
 -\lambda_r & 0 & 0 &  \cdots & \cdots &  \lambda_2 \\
 \end{bmatrix}.
 \label{eq:CtildeProp2_5}
 \end{equation} 
Reformulating we see that the index $j=\lambda_2^{r-3} |v,b_1|^{r-2} \gcd(\lambda_i)_{i\in \mathfrak{I}_{\beta}}$. Let $P_1={\rm Conv}\left(C \right)$, $P_2={\rm Conv}\left(C\cup \left\lbrace 0\right\rbrace \right)$. By Proposition \ref{propn:SubDIagVolume} we have that \begin{equation}
\mu(A,\beta)={\rm Vol}({\rm Conv}\left(C\cup \left\lbrace 0\right\rbrace \right))-{\rm Vol}({\rm Conv}\left(C \right))={\rm Vol}(P_2)-{\rm Vol}(P_1),\label{eq:mu1}
\end{equation}
where ${\rm Vol}$ is the normalized $(r-2)$-dimensional volume. First we compute 
\[ \Vol(P_1) = \pm |v,b_1|^{r-2} \begin{vmatrix}  -\lambda_3 &\lambda_2 & 0 & 0 & \cdots & 0 \\
 -\lambda_4 & 0 & \lambda_2 & 0 &  \cdots & 0 \\
 -\lambda_5 & 0 & 0 & \lambda_2 & \cdots &  0 \\
\vdots & \vdots & \vdots & \vdots & \ddots & \lambda_2  \\ 
 -\lambda_r -\lambda_2 & -\lambda_2 & -\lambda_2 &  \cdots & \cdots &  -\lambda_2 \\ \end{vmatrix}.\]
After doing row and column operations we get that $\Vol(P_1) = |v,b_1|^{r-2} \lambda_2^{r-3} \sum_{i=2}^r -\lambda_i$. Hence after taking the absolute value and normalizing with respect to the index $j$ we get $$\Vol(P_1) = \frac{|\sum_{i=2}^r \lambda_i|}{\gcd(\lambda_i)_{i\in \mathfrak{I}_{\beta}}}. $$
Now consider the polytope $P_2={\rm Conv}\left(C\cup \left\lbrace 0\right\rbrace \right)$. Volume is preserved under taking cones, so we may instead consider the normalized $r-2$ dimensional volume of the convex hull of $$ \tilde{C}= |b_1,v|\cdot \begin{bmatrix}  -\lambda_3 &\lambda_2 & 0 & 0 & \cdots & 0 &0\\
 -\lambda_4 & 0 & \lambda_2 & 0 &  \cdots & 0&0 \\
 -\lambda_5 & 0 & 0 & \lambda_2 & \cdots &  0 &0\\
\vdots & \vdots & \vdots & \vdots & \ddots & 0 &0 \\ 
 -\lambda_r & 0 & 0 &  \cdots & \cdots &  \lambda_2&0 \\
  1 & 1 & 1 &  \cdots & \cdots & 1&1 \\
 \end{bmatrix}.$$ $\tilde{C}$ corresponds to a codimension one toric variety $X_{\tilde{C}}$, by Theorem \ref{theorem:degreeBideal} we have that $${\rm Vol}\left(P_2\right)=\frac{\deg(I_{\tilde{B}})}{|T(\Z^r/\Z {\tilde{B}})|}$$ where $
 \tilde{B}=\begin{bmatrix}
 \lambda_2& \lambda_3& \cdots & \lambda_r& -\sum_{i=2}^r\lambda_i
 \end{bmatrix}
 $ is the Gale dual of $\tilde{C}$. Applying Proposition \ref{proposition:degreeCodim1}, we have that 
$${\rm Vol}\left(P_2\right)=\frac{\max\left( \sum_{i\in v_+^\beta} |\lambda_i|, \sum_{i\in v_-^\beta} |\lambda_i|\right)}{|T(\Z^r/\Z {\tilde{B}})|}.$$
Note that $|T(\Z^r/\Z {\tilde{B}})|$ equals the greatest common divisor of elements of $\tilde{B}$, i.e.~$\gcd(\lambda_i)_{i\in \mathfrak{I}_{\beta}}$.
Substituting the computed values into \eqref{eq:mu1} we have:$$
\mu(P,\beta)=\frac{\max\left( \sum_{i\in v_+^\beta} |\lambda_i|, \sum_{i\in v_-^\beta} |\lambda_i|\right)}{\gcd(\lambda_i)_{i\in \mathfrak{I}_{\beta}}}-\frac{|\sum_{i=2}^r \lambda_i|}{\gcd(\lambda_i)_{i\in \mathfrak{I}_{\beta}}} = \frac{\min\left( \sum_{i\in v_+^\beta} |\lambda_i|, \sum_{i\in v_-^\beta} |\lambda_i|\right)}{\gcd(\lambda_i)_{i\in \mathfrak{I}_{\beta}}}.
$$
Now compute the index $i(P,\beta)$. Consider the submatrix $D$ from \eqref{eq:AinmuPbetaProp},
\[ D= \begin{bmatrix}  [2,r+1] & [1,2] & 0 & 0 & \cdots & 0 & 0 \\
 [2,r+2] & 0 & [1,2] & 0 &  \cdots & 0 & 0 \\
 [2,r+3] & 0 & 0 & [1,2] & \cdots &   \hdots & 0 \\
 \vdots & \vdots & \cdots & \cdots & \ddots & \hdots & 0 \\
[2,n-1] & \vdots & \cdots & \ddots & \cdots & [1,2] & 0  \\ 
 1 & 1 & 1 &  \cdots & \cdots &  1 & 1 \\ \end{bmatrix}. \]
After taking the column-wise Hermite normal form of A, without switching the order of the columns, we will have $(n-2)$ non-zero columns. By Proposition \ref{proposition:LatticeInd1} we have that the index $[Z^{n-2}:\Z A]$ is equal to the determinant of these columns, that is $[Z^{n-2}:\Z A] = [1,2]^{n-4}$. Projecting onto the linear space spanned by $\beta$, we get that $[\Z^{r-2} : \Z A \cap \R\beta]=\frac{[1,2]^{n-4}}{\det D'}$ where $D'$ is the nonzero submatrix of the Hermite normal form of $A$ corresponding to $D$. By Lemma \ref{proposition:indexgcd} $\det(D')$ is equal to the greatest common divisor of the maximal minors of $D$. Consider the inclusion of lattices $\Z^{r-2} \supset \Z A \cap \R\beta \supset \Z \beta,$ we have that
\[ [\Z^{r-2}:\Z \beta] = [\Z^{r-2}: \Z A \cap \R\beta] i(A,\beta). \]
Let $c$ be the greatest common divisor of the maximal minors of $C$ and $d$ be greatest common divisor of the maximal minors of $D$. Then $[\Z^{r-2}: \Z \beta] = c$, thus
 \[ i(P,\beta) = \frac{cd}{[1,2]^{n-4}}. \]
 One computes that $c=[1,2]^{r-3} \gcd([1,i])_{i\in \mathfrak{I}_{\beta}}$ and that $d=[1,2]^{n-r-2} \gcd([2,j])_{j=1}^n$. Thus \begin{align*}
 i(P,\beta) &= \frac{cd}{[1,2]^{n-4}} = \frac{[1,2]^{r-3}\cdot \gcd([1,i])_{i\in \mathfrak{I}_{\beta}}\cdot [1,2]^{n-r-2} \cdot \gcd([2,j])_{j=1}^n}{[1,2]^{n-4}} \\
  &= \frac{ \gcd([1,i])_{i\in \mathfrak{I}_{\beta}} \cdot \gcd([2,j])_{j=1}^n}{[1,2]} = \frac{ \gcd(\lambda_i \det(b_1,v))_{i\in \mathfrak{I}_{\beta}} \cdot  \lambda_2 \cdot \gcd(\det(v,b_j))_{j=1}^n}{ \lambda_2 \det(b_1,v)}  \\
  &=  \gcd(\lambda_i)_{i\in \mathfrak{I}_{\beta}} \cdot  \gcd(\det(v,b_j))_{j=1}^n .
 \end{align*}By Remark \ref{remark:gcdPrimitive} $\gcd(\det(v,b_j))_{j=1}^n = 1$.
\end{proof}

\begin{example}
 In Example \ref{runningexample}, the edge $e_1$ in Figure \ref{figure:S2Ex} is not a simplex, hence by Theorem \ref{theorem:muPrelevant} we have that $i(P,e_1) = \gcd(1,4)=1$, and $\mu(P,e_1) = \min \{ 1,4 \}=1$.
\end{example}

\begin{theorem}
Let $X_A\subset \p^{n-1}$ be a projective toric variety with ${\rm codim}(X_A)=2$ and set $P={\rm Conv}(A)$. Take a face  $\beta$ of $P$ such that not all $b_i \in B_\beta$ are contained in the same relevant line. Let $w_\beta= \sum_{i \in \I_\beta} -b_i$ and let $B'_\beta$ be the matrix obtained by adding $w_\beta$ as an extra row of $B_\beta$. Then 
\[ \mu(P,\beta) = \frac{\deg(I_{B'_\beta})}{|T(\Z^{r+1}/\Z {B'_\beta})|} - \frac{\sum_{j \in \I_\beta | \det(w_\beta,b_j) > 0} \det(w_\beta,b_j)}{|T(\Z^{r+1}/\Z {B'_\beta})|},  \]
\[ i(P,\beta) = |T(\Z^{r+1}/\Z {B'_\beta})| = \gcd([i,j])_{i,j \in \mathfrak{I}_\beta}. \]
\label{theorem:muFromBNoRel}
\end{theorem}
\begin{proof}
Again we use Proposition \ref{proposition:HSAlgorithm}. After rearranging the columns, the matrix has the following form:
\[ A= \begin{bmatrix} 0 & 0 & 0 & 0 & [2,3] & -[1,3] & [1,2] & 0 & 0 & 0  \\
\vdots & \vdots & \vdots & \vdots & [2,4] & -[1,4] & 0 & [1,2] &  \ddots & 0 \\
0 & 0 & 0 & 0& [2,5] & -[1,5] & 0 & 0 & \ddots & [1,2]  \\
[1,2] & 0 & 0 & 0 & \vdots & \vdots & 0 & 0 & \ddots & 0  \\ 
0 & [1,2] & 0 & \vdots & \ddots & \ddots & \ddots & \ddots & \ddots &  \vdots \\
\vdots & \vdots & \ddots & 0 & [2,n-1] & -[1,n-1] & 0 & 0 &  \cdots & 0  \\
1 & 1 & 1 & 1 & 1 & 1 & 1 & 1 & 1 & 1  \\\end{bmatrix}. \]
Note that, by Proposition \ref{proposition:faceDimension}, $\beta$ has codimension $r-2$. Now consider the matrix $A$, the only row operations we need to perform to pick out the correct submatrix $C$ in Proposition \ref{proposition:HSAlgorithm} is exchanging the top and bottom rows. Let $\{ c_1,...,c_r \}$ denote the columns of the resulting $(r-2) \times r$ submatrix $C$.
To compute the subdiagram volume $\mu(P,\beta)$ we first compute $\Vol(\Conv(c_1,...,c_r,0))$. Consider the $(r-1) \times (r+1)$ matrix   matrix $A'$ with rows $a_1',...,a_{r+1}'$ of the form
\[ A'= \begin{bmatrix} c_1 & c_2 & \cdots & c_r & 0 \\
1 & 1 & \cdots  & 1 & 1 
\end{bmatrix}. \]
Observe that by construction $A'$ has rank $r-1$ and that 
$B'_\beta \subset \ker A'$. We want to compute $\Vol(\Conv(c_1,...,c_r,0))$ normalized with respect to the lattice spanned by $C$. We have that $$\Vol(\Conv(c_1,...,c_r,0)) = \Vol(\Conv(a_1',...,a_{r+1}')),$$ since the second convex hull is equivalent to taking the cone over $\Conv(c_1,...,c_r,0)$ and normalized volume is preserved under taking cones.

$B'_\beta$ generates a lattice ideal $I_{B'_\beta}$ of codimension two. Since $B'_\beta \subset \ker(A')$ applying Theorem \ref{theorem:degreeBideal} gives
\[ \Vol(\Conv(c_1,...,c_r,0))=\Vol(\Conv(a_1',...,a_{r+1}')) = \frac{\deg I_{B'_\beta}}{[\Z^{r+1} : \Z {B'_\beta}]} .\]

Now consider the volume of the convex hull of $\{ c_1,...,c_r \}$, normalized with respect to the lattice spanned by $C$. Let $A''$ be the $(r-1) \times r$ matrix obtained by adding the row $(1,...,1)$ to $C$. There are two cases. If $(1,...,1)$ is already contained in the row span of $C$ then all lattice points $c_i$ are contained in an affine hyperplane. Dimension considerations dictate that the resulting normalized volume is zero. Note that in this case we automatically get by Gale duality that $w_\beta=0$, thus verifying this result. If $(1,...,1)$ is not in the row span of $C$ then $A''$ has rank $r-1$. Consider the $ 1 \times r$ matrix $B''$ with rows $\det(w_\beta,b_i)$. Then $B''$ is contained in $\ker (A'')$. We have that $\Vol(\Conv(c_1,...,c_r))=k\Vol(\Conv(0,a_1'',...,a_{r}''))$ where\begin{equation}
 k = \frac {[\Z^{r-1}: \Z A'']}{[\Z^{r-2}: \Z C]}.\label{eq:k_ind3_3}
\end{equation}
The equality above follows form the fact that, for a polytope $\mathcal{P}$ of dimension $\mathfrak{n}$, the $\mathfrak{n}$-dimensional Euclidean volume of $\mathcal{P}$ is equal to the $\mathfrak{n}+1$-dimensional Euclidean volume of a pyramid of height one over $\mathcal{P}$. The integer $k$ in \eqref{eq:k_ind3_3} arises since the lattice indices of the polytopes generated by the points $c_i$ and the points $a_i''$ may differ.

We claim that $[\Z^{r-1}: \Z A'']= [1,2]^{r-3} \gcd(\det(w_\beta,b_i))_{i \in \I_\beta}$  and $[\Z^{r-2}: \Z C]=[1,2]^{r-3} \gcd([i,j])_{i,j \in \I_\beta}$ hence 
\[ k = \frac{\gcd(\det(w_\beta,b_i))_{i \in \I_\beta}}{\gcd({[i,j])_{i,j \in \I_\beta}}}. \]
Indeed $[\Z^{r-2}: \Z C]$ has to equal $[\Z^{r-1} : \Z A']$ which we know, by the proof of Proposition \ref{proposition:latticeIndexA}, is equal to $[1,2]^{r-3} \gcd([i,j])_{i,j \in \I_\beta}$. The claim for the last of the indices can be proved  using elementary row operations; $A''$ is a $r-1 \times r$ matrix. Denote by $m_i$ the maximal minor obtained by deleting column $i$. We see that
\[ m_1 = [1,2]^{r-3} ([1,2]+[1,3]+ \cdots + [1,r])=[1,2]^{r-3} \det (1,\sum_{i=1}^r b_i) = [1,2]^{r-3} \det(w_\beta,b_i). \]
Similarly $m_2 = [1,2]^{r-3} \det(b_2,w_\beta)$. For $k>2$ we see that after column and row operations 
\[ m_k = \pm [1,2]^{r-4} \left( [2,k] \sum_{i=1 \; i \neq k}^r [1,i] + [1,k] \sum_{i=1 \; i \neq k}^r [i,2] \right),\]
which after substituting for $w_\beta$, rearranging and using the Plucker relation, equals $ \pm [1,2]^{r-3}\det(b_k,w_\beta)$. This proves the claim that $[\Z^{r-1}: \Z A'']= [1,2]^{r-3} \gcd(\det(w_\beta,b_i))_{i \in \I_\beta}$. Using this and Theorem \ref{theorem:degreeBideal} we get that\begin{align*}
 \Vol(\Conv(c_1,...,c_r))&=\Vol(\Conv(0,a_1'',...,a_{r}'')) \frac{\gcd(\det(w_\beta,b_i))_{i \in \I_\beta}}{\gcd([i,j])_{i,j \in \I_\beta}} \\
&= \frac{\deg I_{B''}}{|T(\Z^{r+1}/\Z {B''})|} \frac{\gcd(\det(w_\beta,b_i))_{i \in \I_\beta}}{\gcd([i,j])_{i,j \in \I_\beta}} .
\end{align*}
Now $|T(\Z^{r+1}/\Z {B''})| = \gcd(\det(w_\beta,b_i))_{i \in \I_\beta}$ and by Proposition \ref{proposition:degreeCodim1} we have that $$\deg I_{B''}=\sum_{j \in \I_\beta, \det(w_\beta,b_j) > 0} \det(w_\beta,b_j),$$ hence
\[ \Vol(\Conv(c_1,...,c_r))=\frac{\sum_{j \in \I_\beta, \det(w_\beta,b_j) > 0} \det(w_\beta,b_j)_{i \in \I_\beta}}{\gcd([i,j])_{i,j \in I_\beta}} .\]
Finally it remains to prove that $i(A,\beta)=|T(\Z^{r}/\Z {B'_\beta})|$ (which equals $\gcd([i,j])_{i,j \in \I_\beta}$). By considering the block form of $A$ from above we see that the lattice points of $\beta$ generate a lattice of index $[1,2]^{n-r-1}$. We also know that $[\Z^n : \Z A]=[1,2]^{n-4}$. We get that 
\[ [1,2]^{n-4}i(A,\beta)= [1,2]^{n-r-1} [\Z^{r-2} : \Z C]. \]
Since $[\Z^{r-2} : \Z C]=[1,2]^{r-3}|T(\Z^{r}/\Z {B'_\beta})|$ we obtain the desired result.

\end{proof}

\begin{example}\label{runningexample5}
In Example \ref{runningexample}, the edge $e_3$ in Figure \ref{figure:S2Ex} is a simplex, by Theorem \ref{theorem:muFromBNoRel} we have 
\[ i(P,e_3) = \gcd (8,-8,12)=4 , \;\; {\rm and}\;\;\; \mu(P,e_3) = \frac{12}{4} - \frac{4}{4} = 2 \]
where the last number is obtained as follows. The matrix $B'_{e_3}$ from Theorem \ref{theorem:muFromBNoRel} is
\[B'_{e_3}=\begin{bmatrix}  2 & 2 \\ -4 & 0 \\ 1 & -3  \\ 1 & 1 \end{bmatrix}\]
and the vector $w_{e_3}$ equals $(1,1)$. The degree of $I_{B'_{e_3}}$ is $12$ and is computed using Proposition \ref{proposition:degreeBmatrix}. Also note that in this case we have $( \det(w_{e_3},b_3),\det(w_{e_3},b_4),\det(w_{e_3},b_5))=(0,4,-4)$. We may now compute the Euler obstructions $\Eu(e_i)=i(P,e_i) \mu(P,e_i)$ for any edge $e_i$, the results are summarized in Table \ref{table:faces}. 

To find the Euler obstruction of the vertex $v_1$ we will again apply Theorem \ref{theorem:muFromBNoRel} to compute $\mu(P,v_1)$. For any vertex $i(P,\alpha)$ is automatically $1$. The matrix $B'_{v_1}$ is in this case just $B$ itself, which has degree $12$ (see Example \ref{runningexampleBdeg}). We have that $w_{v_1}$ equals $(1,0)$ hence $(\det(w_{v_1},b_2),\det(w_{v_1},b_3),\det(w_{v_1},b_4)\det(w_{v_1},b_5))= (1,2,0,-3),$ giving
\[ \mu(P,v_1) = 12 - 3 = 9.\]To complete the computation of the Euler obstructions of the vertices (using Proposition \ref{propn:EUPropDef}) we also need to compute $i(e_i,v_j)$ and $\mu(e_i,v_j)$. We develop the necessary tools to do this using the matrix $B$ in \S\ref{subsection:mu_alpha_beta} below. 
\end{example}

\subsubsection{Subdiagram volumes $\mu(\alpha,\beta)$}
\label{subsection:mu_alpha_beta}
We now consider the subdiagram volumes of two proper faces of the polytope associated to a codimension two projective toric variety. 
\begin{theorem}
Let $X_A\subset \p^{n-1}$ be a projective toric variety with ${\rm codim}(X_A)=2$, set $P={\rm Conv}(A)$, and let $\beta \subset \alpha$ be faces of $P$. We have that \[ \mu(\alpha,\beta) = 1, \;\; {\rm and}\;\;\; i(\alpha,\beta) = 1 \] if any of the following conditions hold:\begin{enumerate}
\item[(i)] not all rows of $B_\alpha$ are contained in a relevant line,
\item[(ii)] all rows of $B_\beta$ are contained in the same relevant line.
\end{enumerate} 
\label{theorem:twoFaces1}
\label{theorem:twoFaces2}
\end{theorem}\begin{proof}First consider case (i), where not all rows of $B_\alpha$ are contained in a relevant line. In this case both $\beta$ and $\alpha$ are simplices; this means that the $\dim(\alpha)$-dimensional volume of $\alpha\backslash \beta$ is zero, and hence $\mu(\alpha,\beta) ={\rm Vol}({\rm Conv}(A_{\alpha}))=1$, since $\alpha$ is a simplex. This concludes the proof of (i).

Now consider consider case (ii), where all rows of $B_\beta$ are contained in the same relevant line. Set $r=|\I_\beta|$ and $s=|\I_\alpha|$. To calculate $\mu(\alpha,\beta)$ we are must pick out the correct submatrix $C$ in Proposition \ref{propn:SubDIagVolume} and compute the resulting normalized volumes. By our assumption on the face structure of $\alpha$ and $\beta$ the correct submatrix will be an $r-s \times r-s$ matrix of full rank. But the convex hull of $r-s$ points in $r-s$ dimensional space has volume 0, meaning the second term in the expression for $\mu(\alpha,\beta)$ in Proposition \ref{propn:SubDIagVolume} is zero. Thus $\mu(\alpha,\beta)$ is equal to the volume of the convex hull of the columns of $C$ after we add $0$. By the argument above this is a simplex. The volume of a simplex (inside the lattice spanned by the submatrix) equals one.

The lattice points of $\alpha$ span a linear subspace of dimension $n-s-2$. The lattice points of $\beta$ span a linear subspace of dimension $n-r-2$. The lattice points of $\alpha \setminus \beta$ are $s-r$ lattice point which span a linear subspace of dimension $s-r$. Hence each $a_i \in \alpha \setminus \beta$ is part of a basis of the lattice $L$ generated by $\alpha$. Thus any lattice point of  $L$ which also lies in $L'$ has to be in the lattice generated by $\beta$. It follows that the index $i(\alpha, \beta)=1$.
\end{proof}

\begin{theorem}
Let $X_A\subset \p^{n-1}$ be a projective toric variety with ${\rm codim}(X_A)=2$ and set $P={\rm Conv}(A)$. Consider faces $\beta \subset \alpha$ of $P$ where not all rows of $B_\beta$ are contained in the same relevant line, but all rows $B_\alpha$ are contained in the same relevant line. Let $v$ be primitive vector of the relevant line containing the rows of $B_\alpha$ and let $\gamma_i=\det(v,b_i)$, for $i \in \I_\beta$. Then \[\mu(\alpha,\beta) =\frac{\min \left( \sum_{i\;:\; \gamma_i>0}|\gamma_i|, \sum_{i\;:\; \gamma_i<0}|\gamma_i|\right)}{\gcd (\gamma_i)_{i \in \I_\beta}},\; \; {\rm and}\;\;\;i(\alpha,\beta) = \gcd (\gamma_i)_{i \in \I_\beta} . \]\label{theorem:twoFaces3}
\end{theorem}
\begin{proof}
This proof of this result is very similar to that of Theorem \ref{theorem:muPrelevant}.  We may order the rows of the matrix $B$ (whose Gale dual defines $P$) so that $B_\alpha=\left\lbrace b_1,b_3,\dots, b_{r+1}\right\rbrace$ and $B_\beta=\left\lbrace b_1, b_2,\dots,b_s\right\rbrace$ where $b_2$ is not in the relevant line $v$. Then we have that $[1,3]=[1,4]=...=[1,r+1]=0$ since they are all contained in a relevant line. Hence we see that the correct submatrix $C$ in Proposition \ref{propn:SubDIagVolume} is given by
\[ C= \begin{bmatrix}  -[1,r+2] & [1,2] & 0 & 0 & \cdots & 0 \\
 -[1,r+3] & 0 & [1,2] & 0 &  \cdots & 0 \\
 -[1,r+4] & 0 & 0 & [1,2] & \cdots &  0 \\
\vdots & \vdots & \vdots & \vdots & \ddots & 0  \\ 
 -[1,s] & 0 & 0 &  \cdots & \cdots &  [1,2]  \\ \end{bmatrix} .\]
This matrix has the same form as the matrix in \eqref{eq:CtildeProp2_5}, hence the remainder of the proof proceeds similarly to that of Theorem \ref{theorem:muPrelevant}. In the resulting formula we obtain sums over $i$ such that $b_i \in B_\beta \setminus B_\alpha$, however since $\gamma_i=0$ when $i \in B_\alpha$ the same formula is true when looping over all rows of $B_\beta$.

\end{proof}
\begin{example}
Resuming Example \ref{runningexample} and Example \ref{runningexample5} we may now complete the computation of the Euler obstruction ${\rm Eu}(v)$ for a vertex $v$ in Figure \ref{figure:S2Ex}. To compute the Euler obstructions of the vertices we need to compute the numbers $i(e_i,v_j)$ and $\mu(e_i,v_j)$ using Theorems \ref{theorem:twoFaces1}, \ref{theorem:twoFaces2}, and \ref{theorem:twoFaces3}. Since $v_j$ is a vertex we have that $i(e_i,v_j)=1$ whenever it is defined. Table \ref{table:muve} gives the numbers $\mu(e_i,v_j)$. Putting this together we have that \[\Eu(v_1) = - \mu(P,v) + \mu(e_2,v_1) \Eu(e_2) + \mu(e_3,v_1) \Eu(e_3) = -9 +1 \cdot 1 + 1 \cdot 8 = 0.\] The Euler obstructions of the other vertices are computes similarly and are summarized in Table \ref{table:faces}.
\begin{table}[ht]
\begin{tabular}{c | c c c}
$\mu(e_i,v_j)$ & $v_1$ & $v_2$ & $v_3$ \\
\hline
$e_1$ & $\ast $ & 2 & 1 \\
$e_2$ & 1 & $\ast$ & 1 \\
$e_3$ & 1 & 1 & $\ast$ \\
\end{tabular}
\caption{Subdiagram volumes $\mu(e_i,v_j)$. We write $\ast$ when there is no containment relation between $e_i$ and $v_j$, i.e.~when $\mu(e_i,v_j)$ is undefined. \label{table:muve}}
\end{table}
\end{example}

\subsection{Volume Calculation Via the Gale Dual in Codimension two}
In this subsection we consider the problem of computing the volume of faces of the polytope associated to a codimension two projective toric variety $X_A$ using the Gale dual $B$ of $A$.

\begin{proposition}
Let $X_A\subset \p^{n-1}$ be a projective toric variety with ${\rm codim}(X_A)=2$ and set $P={\rm Conv}(A)$. Let $\beta$ be a face such that all rows of $B_\beta$ are contained in a relevant line. Let $v$ be a primitive vector in the relevant line. Then
\[ \Vol(\beta) = \sum_{j \in \I_\beta^c | \det(v,b_j) > 0} \det(v,b_j).  \]
\end{proposition}
\begin{proof}

The lattice points of $\beta$ defines a toric variety $X_\beta$ of codimension $1$. We have that $\deg X_\beta= \Vol(\beta)$. By Proposition \ref{proposition:degreeCodim1} we have that
\[ \deg X_\beta = \sum_{j \in \I_\beta^c | \det(v,b_j) > 0} \det(v,b_j). \]
\end{proof}
We may now fill in all values in Table \ref{table:faces} using only the Gale dual matrix $B$ and obtain the ED degree and polar degrees of the variety $X_A$ from Example \ref{runningexample}. The results detailed in this section allow us to compute the ED degree and polar degrees of much larger examples much faster, as discussed in \S\ref{section:computation} below.

Finally we remark that we would have hoped to solve the recursion for the Euler obstruction and find compact formulas for the polar degrees and the ED degree of a projective toric variety of codimension two, similar to the codimension one case \cite[Theorem 3.7]{HS}. Unfortunately we have not been able to simplify the expressions sufficiently to find satisfying formulas in the codimension two case. While the lack of such formulas has little effect on the computational performance of the codimension two methods finding them would be mathematically appealing.

\section{Computational Performance}
In this section we briefly compare the computational performance of the specialized codimension two methods developed in \S\ref{section:MainResults} which use the Gale dual matrix $B$ with the performance of the general purpose (i.e.~for any codimension) $A$-matrix methods described in \cite{HS}. We will refer to these as the ``$B$-matrix method'' and the ``$A$-matrix method'', respectively. 

\begin{table}[h!]
\centering
\resizebox{\linewidth}{!}{
\begin{tabular}{@{} *6c @{}}
\toprule 
 \multicolumn{1}{c}     {\color{Ftitle} Example}  &   {\color{Ftitle}Size $A$} &   {\color{Ftitle}$A$-matrix method \cite{HS}} & {\color{Ftitle}$B$-matrix method \S\ref{section:MainResults}} &   {\color{Ftitle} Find faces}  &   {\color{Ftitle} Speedup factor} \\ 
 \midrule 
  \color{line}$X_{A_1}$ & $4\times 6$& \color{line}4.6s& \color{line} 0.2s& \color{line}0.3s& \color{line}23.0 \\ 
      \color{line}$X_{A_2}$ & $4\times 6$& \color{line}6.1s& \color{line} 0.2s& \color{line}0.3s& \color{line} 30.5\\ 
    \color{line}$X_{A_3}$ & $6\times 8$& \color{line}106.1s& \color{line} 1.6s& \color{line}2.3s& \color{line} 66.3\\ 
        \color{line}$X_{A_4}$ & $6\times 8$& \color{line}118.4s& \color{line} 1.8s& \color{line}2.7s& \color{line} 65.8\\ 
         \color{line}$X_{A_5}$ & $7\times 9$& \color{line}495.5s& \color{line} 5.3s& \color{line}8.1s& \color{line} 93.5\\ 
  \color{line}$X_{A_6}$ & $8\times 10$& \color{line}2611.5s& \color{line} 26.7s& \color{line}37.7s& \color{line} $ 97.8$\\ 
\bottomrule
 \end{tabular}}\vspace{1mm}
\caption{Average run times to compute the polar degrees of a codimension two projective toric variety $X_A$. The run time to generate the face lattice of ${\rm Conv}(A)$ is listed separately in the fifth column since both methods must perform this step (hence the total run time is the sum the time to find the faces and the time of either the $A$-matrix or $B$-matrix method).} \label{table:runtimes}
 \end{table}
\label{section:computation}

When computing the polar degrees or ED degree of a projective toric variety using \eqref{eq:polarDegressCM}, with either the $A$-matrix or $B$-matrix method, the primary computational task is to compute the Chern-Mather volumes of $X_A$. While the number of steps in the recursive loops for both the $A$ and $B$ matrix methods is the same the computational cost of computing the subdiagram volumes $\mu(\alpha,\beta)$ differs quite substantially. We will focus on analyzing the cost of this computation in the case where ${\rm codim}(X_A)=2$ (i.e.~where the methods of \S\ref{section:MainResults} are applicable).

\subsection{Computational Cost of $\mu(\alpha,\beta)$ for General $A$}
First consider the $A$-matrix method. Let $X_A$ be a projective toric variety with $P={\rm Conv}(A)$ and let $\alpha \supset \beta$ be faces of $P$. Further suppose that the face lattice has already been computed and that the relevant Hermite normal forms (needed for Proposition \ref{propn:SubDIagVolume}) have been stored in the process. Let $r={\dim(\alpha)-\dim(\beta)}$. To compute the subdiagram volume (using \eqref{eq:subdiagramVol} or Proposition \ref{propn:SubDIagVolume})
we must compute the following things:\begin{itemize}
\item the convex hulls of two collections of at least $r+1$ lattice points in $\R^{r}$
\item the volumes of two dimension $r$ polytopes. 
\end{itemize}For computing the convex hull of $ m$ points in $\R^r$ there exist known (optimal) algorithms of complexity $$O\left(m\log(m)+m^{\lfloor\frac{r}{2}\rfloor}\right),$$ see \cite{chazelle1993optimal}. Calculating the volume of a polytope in dimension $r$ is known to be a $\#P$-hard problem \cite[Theorem~1]{dyer1988complexity}. For existing algorithms (to the best of our knowledge) there is not a known compact (i.e.~readable/meaningful) run time bound for finding the dimension of an arbitrary polytope in dimension $r$ and different algorithms may vary from being exponential to factorial in $r$ for different polytopes. Using known algorithms, the computational cost of computing the volume of the dimension $r$ hypercube varies from being approximately factorial in $r$, i.e.~$O(r!)$, to being approximately exponential in $r$, i.e.~$O(r^24^r)$, depending on the algorithm chosen. See \cite{bueler2000exact} for an in depth discussion of current algorithms. Hence, in particular, the cost of computing the subdiagram volume $\mu(\alpha,\beta)$ will be (at least) exponential, possibly factorial, in the relative dimension $r=\dim(\alpha)-\dim(\beta)$. Further note that $r$ may be as large as $d-2$ for a $d\times n$ integer matrix $A$.

\subsection{Computational Cost of $\mu(\alpha,\beta)$ using \S\ref{section:MainResults} when ${\rm codim}(X_A)=2$ }Let $X_A$ be a codimension two projective toric variety with $P={\rm Conv}(A)$ and let $\alpha \supset \beta$ be faces of $P$. We again suppose that the face lattice has already been computed, and that faces contained in relevant lines have been identified during this process (this includes the computation of the scaling factors $\lambda$ of each vector $b=\lambda v$ for $v$ the primitive vector defining the relevant line). By precomputing all lattice indices $[\Z^d:\Z A_{\alpha}]$ for each face $\alpha$ of $P$ we may compute many of the expressions $\mu(\alpha,\beta)$ in constant time (i.e.~only a table lookup is needed). Note that, using Proposition \ref{proposition:LatticeInd1}, the computation of the lattice index $[\Z^d:\Z A_{\alpha}]$ for each face requires the computation of one Hermite normal form and one determinant of the resulting square matrix; many efficient algorithms exist for these computations. Assuming the above precomputations, the number of operations required to compute $\mu(\alpha,\beta)$ using the methods of \S\ref{section:MainResults} is as follows:
\begin{itemize}
\item constant if $\alpha$ is a proper face of $P$ and either all lattice points defining $\beta$ are in a relevant line or if neither $\alpha$ nor $\beta$ is contained in a relevant line
\item linear in $\ell_{\beta}$, where $\ell_{\beta}$ is the number of rows of $B_{\beta}$, if $\alpha$ is a proper face of $P$ where all rows of $B_{\alpha}$ are contained in a relevant line but all rows of $B_{\beta}$ are not
\item proportional to the number of operations to compute a greatest common divisor of $\ell_\beta$ integers if $\alpha=P$ and all rows of $B_{\beta}$ are contained in a relevant line
\item quadratic in $\ell_{\beta}$ if $\alpha=P$ and the rows of $B_{\beta}$ are not contained in a relevant line (in this case we are summing $2\times 2$ determinants)
\end{itemize}In particular we see that for the vast majority of possible pairs of faces $\alpha\supset \beta$ the cost of computing $\mu(\alpha,\beta)$ will be linear or constant, and at worst will be polynomial in the number of rows of $B_{\beta}$, which will always be less than or equal to the number of rows of $B$.

\subsection{Summary}Suppose $A$ is a $(n-2)\times n$ integer matrix defining a codimension $2$ projective toric variety $X_A$. Examining the definitions of the Chern-Mather volume and Euler obstruction we see that to compute all Chern-Mather volumes (and hence to compute the ED degree or polar degrees) we must compute the subdiagram volume $\mu(\alpha,\beta)$ for all possible pairs of faces $\alpha \supset \beta$ of $P={\rm Conv}(A)$. Let $F$ denote the number of faces of $P$, there are $\frac{F^2-F}{2}$ such pairs. Using the general purpose $A$-matrix methods we must perform computations which are at least exponential, possibly factorial, in $\dim(\alpha)-\dim(\beta)$; for the majority of pairs this will mean computations that are \textit{at least} exponential in a number larger than $\frac{n}{3}$. 

Let $\ell_{\beta}$ denote the number of rows in the matix $B_{\beta}$; this number will always be less than or equal to $n$. With the specialized $B$-matrix methods of \S\ref{section:MainResults} the computation of $\mu(\alpha,\beta)$ will be done in constant or linear time relative to $\ell_{\beta}$ for $\frac{F^2-3F}{2}$ of the pairs with the remaining $F$ pairs being done in at most quadratic, $O(\ell_{\beta}^2)$, time relative to the number of rows of $B_{\beta}$. 

In light of the discussion above the significant runtime gains yielded by the $B$-matrix methods, as displayed in Table \ref{table:runtimes}, are not surprising. It should, however, be noted that both combinatorial methods, either $A$-matrix or $B$-matrix, will be able to compute ED degrees and other invariants for projective toric varieties $X_A$ which simply would not be feasible using other current methods. For example ${\rm EDdegree}(X_{A_5})=301137686$ (see Table \ref{table:degrees}) represents the degree of the variety defined by the critical equations of the Euclidean distance function for $X_{A_5}$. Computing this number using algebraic/geometric methods (i.e.~Gr\"obner basis, numerical algebraic geometry, etc.) would require finding the degree of a zero dimensional variety consisting of \textit{greater than} $300$ \textit{million} isolated points in $\p^8$. Such a computation would be infeasible with current algebraic/geometric methods even over a span of weeks running on a super computer whereas the $A$-matrix or $B$-matrix methods compute this number (on a laptop) in a matter of minutes or seconds, respectively.  
\bigskip \bigskip

\begin{small}

\noindent
{\bf Acknowledgements.}
We are grateful to Ragni Piene and John Christian Ottem for their helpful comments on this paper. We would also like to thank Bernd Sturmfels for suggesting the topic to us. Martin Helmer was supported by an NSERC (Natural Sciences and Engineering Research Council of Canada) postdoctoral fellowship during the preparation of this work.
\end{small}

\bigskip
\appendix
\section{Computational Examples List}\label{appendix:compEx}
In this appendix we list the integer matrices $A$ defining the codimension two toric varieties listed in Table \ref{table:runtimes}.\small$$
A_1=\begin{bmatrix}10&
      1&
      0&
      {-7}&
      0&
      0\\
      {-7}&
      0&
      1&
      5&
      0&
      0\\
      2&
      0&
      0&
      0&
      1&
      0\\
      {-4}&
      0&
      0&
      3&
      0&
      1\\
      \end{bmatrix}
$$
$$
A_2=\begin{bmatrix}3&
      0&
      0&
      1&
      1&
      2\\
      3&
      5&
      0&
      2&
      1&
      3\\
      0&
      1&
      2&
      0&
      2&
      0\\
      1&
      1&
      1&
      1&
      1&
      1\\
      \end{bmatrix}
$$
$$
A_3=\begin{bmatrix}0&
      0&
      0&
      {-31}&
      0&
      {-7}&
      1&
      {-31}\\
      0&
      0&
      0&
      {-12}&
      0&
      {-2}&
      0&
      {-11}\\
      0&
      0&
      1&
      {-2}&
      {-1}&
      0&
      0&
      {-2}\\
      0&
      {-1}&
      0&
      1&
      1&
      0&
      0&
      1\\
      {-1}&
      0&
      0&
      7&
      0&
      1&
      0&
      7\\
      0&
      1&
      0&
      13&
      0&
      3&
      0&
      13\\
      \end{bmatrix}
$$

$$
A_4=\begin{bmatrix}3&
      0&
      0&
      1&
      1&
      2&
      1&
      2\\
      3&
      5&
      0&
      2&
      1&
      3&
      12&
      11\\
      5&
      1&
      9&
      10&
      12&
      3&
      7&
      9\\
      3&
      1&
      2&
      19&
      7&
      1&
      1&
      2\\
      0&
      1&
      2&
      0&
      2&
      0&
      5&
      7\\
      1&
      1&
      1&
      1&
      1&
      1&
      1&
      1\\
      \end{bmatrix}
$$
$$A_5=
\begin{bmatrix}3&
      0&
      0&
      1&
      1&
      2&
      1&
      2&
      7\\
      3&
      5&
      0&
      2&
      1&
      3&
      12&
      11&
      12\\
      5&
      1&
      9&
      10&
      12&
      3&
      7&
      9&
      3\\
      3&
      1&
      2&
      19&
      7&
      1&
      1&
      2&
      1\\
      0&
      1&
      2&
      0&
      2&
      0&
      5&
      7&
      21\\
      3&
      1&
      5&
      11&
      22&
      10&
      15&
      0&
      1\\
      1&
      1&
      1&
      1&
      1&
      1&
      1&
      1&
      1\\
      \end{bmatrix}
$$
$$A_6=\begin{bmatrix}2&
      3&
      4&
      0&
      {-1}&
      {-2}&
      5&
      9&
      7&
      0\\
      13&
      10&
      {-2}&
      21&
      {-1}&
      2&
      5&
      2&
      1&
      4\\
      1&
      3&
      1&
      0&
      {-2}&
      21&
      31&
      2&
      1&
      2\\
      7&
      15&
      11&
      3&
      4&
      2&
      6&
      7&
      8&
      1\\
      14&
      2&
      3&
      1&
      9&
      12&
      {-1}&
      {-1}&
      {-2}&
      {-1}\\
      1&
      {-1}&
      {-2}&
      0&
      2&
      0&
      4&
      7&
      {-6}&
      15\\
      31&
      11&
      0&
      5&
      1&
      {-2}&
      4&
      5&
      0&
      {-1}\\
      1&
      1&
      1&
      1&
      1&
      1&
      1&
      1&
      1&
      1\\
      \end{bmatrix}$$

\normalsize
For reference we include the degree, the degree of the dual variety, and ED degree of the toric varieties defined by the matrices above in Table \ref{table:degrees}. 
\begin{table}[h!]
\centering
\resizebox{.45\linewidth}{!}{
\begin{tabular}{@{} *6c @{}}
\toprule 
 \multicolumn{1}{c}     {\color{Ftitle} Example}  &   {\color{Ftitle}$\deg(X_A)$} &   {\color{Ftitle}$\deg((X_A)^{\vee})$} &  {\color{Ftitle}${\rm EDdegree}(X_A)$} \\ 
 \midrule 
    \color{line}$X_{A_1}$ & \color{line}$19$& \color{line}$27$& 
  \color{line} $170$\\
   \color{line}$X_{A_2}$ & \color{line}$28$& \color{line}$45$& 
  \color{line} $252$\\
    \color{line}$X_{A_3}$ & \color{line}$70$& \color{line}$125$& 
  \color{line} $2356$\\
  \color{line}$X_{A_4}$ & \color{line}$16924$& \color{line}$30840$& 
  \color{line} $641134$\\
  \color{line}$X_{A_5}$ & \color{line}$4570434$& \color{line}$8222171$& 
  \color{line} $301137686$\\
    \color{line}$X_{A_6}$ & \color{line}$ 581454473$& \color{line}$1056983492$& 
  \color{line} $74638158177$\\
\bottomrule
 \end{tabular}}\vspace{1mm}
\caption{The degree, degree of the $A$-discriminant, and the ED degree of the projective toric varieties appearing in Table \ref{table:runtimes}.  \label{table:degrees}}
 \end{table}
 
 \section{Formulas expressed in alternate index convention}

\label{section:AltIndConv}
In this subsection we restate the results of \S\ref{subsection:mu_A_beta} and \S\ref{subsection:mu_alpha_beta} (which used the index convention of \cite{BIUN,MT}) above in the index convention of \cite{HS}. In \cite{HS} the lattice index is contained in the normalized volume $\Vol$, while in the convention of \cite{BIUN,MT} it is contained in Euler obstruction $\Eu$. More precisely we let

\begin{align*}
\Eu'(\beta) &= \frac{Eu(\beta)} {i (A,\beta)} \\
\Vol'(\beta) &= \Vol(\beta)i(A,\beta) \\
\mu'(\alpha,\beta) &= \mu(\alpha,\beta) \frac {i(\alpha,\beta) i(A,\alpha)}{i(A,\beta)}.
\end{align*}

Let $A$ be a $d\times n$ integer matrix with $(1,\dots,1)$ in its row space, let $B$ be the Gale dual of $A$ and let $P={\rm Conv}(A)$. In this subsection we assume that $[\Z^d:\Z A]=1$. We can make this assumption without loss of generality since for the purposes of computing the polar degrees we may always find such an $A$ defining a toric variety isomorphic to the original one.

The index convention presented in this subsection is more convenient for volume computations, since we do not have to explicitly compute lattice indexes. This conversion also gives a cleaner expression for the Euler obstruction of a face since the expressions $i(\alpha, \beta)$ do not appear in the formula. More precisely the {{\em Euler obstruction}}, of a face $\beta$ of $P$, is given by 
\begin{enumerate}
\item ${\rm Eu'}(P)\,\,=\,\,1,$
\item $\displaystyle {\rm Eu'}({\beta}) \quad =\sum_{{\alpha} {\rm \;s.t.}\; {\beta}{\;\rm is\; a} \atop                                                        
{\rm proper\; face \; of \;} {\alpha}}                                          
\!\!\!\!                                                                        
(-1)^{\dim({\alpha})-\dim({\beta})-1}\cdot { \mu'(\alpha, \beta)} \cdot {\rm Eu'}(\alpha).$
\end{enumerate}As before the dimension $i$ Chern-Mather volume is given by $V_i=\sum_{\dim(\alpha)=i}{\rm Vol'}(\alpha){\rm Eu'}(\alpha)$. 
The formulas for the the polar degrees \eqref{eq:polarDegressCM}, and hence the ED degree \eqref{eq:sumofpolar}, in terms of the Chern-Mather volumes remain unchanged. 

We now restate the expressions for $\mu(\alpha,\beta)$ given in the previous subsections \S\ref{subsection:mu_A_beta} and \S\ref{subsection:mu_alpha_beta} in terms of this index convention. In the propositions below we let $A$ be a $d\times n$ integer matrix of full rank defining a projective toric variety $X_A$ and let $B$ be the Gale dual of $A$. 

In Proposition \ref{prop:muFromBNoRel_cov2} the expression $\deg(I_{B'})$ is computed using Proposition \ref{proposition:degreeBmatrix}.
\begin{proposition} \label{proposition:muPrelevant_cov2}
Let $X_A\subset \p^{n-1}$ be a projective toric variety with ${\rm codim}(X_A)=2$ and set $P={\rm Conv}(A)$. Let $\beta$ be a proper face of $P$. Then the subdiagram volume $\mu'(P,\beta)$ is as follows:\begin{enumerate}
\item[(a)] if all rows of $B_\beta$ are contained in a relevant line $v$ with $b_i=\lambda_iv$ for $b_i$ a row of $B_{\beta}$ then, if $v_+$ indices $\lambda_i>0$ and $v_-$ indices $\lambda_i<0$, we have \[
\mu'(P,\beta) =  \frac{\min\left( \sum_{i\in v_+} |\lambda_i|, \sum_{i\in v_-} |\lambda_i|\right)}{\gcd(\lambda_i)_{i \in \I_\beta}},
\]
\item[(b)] if all rows of $B_\beta$ are not contained in a relevant line  then 
\[ 
\mu'(P,\beta) =  \frac{\deg(I_{B'_\beta})}{|T(\Z^n/ \Z {B'_\beta})|} - \frac{\sum_{j : \det(w_\beta,b_j) > 0} \det(w_\beta,b_j)}{|T(\Z^n/ \Z {B'_\beta})|}. 
\] where $B'_\beta$ is the matrix $B_\beta$ with the row $w_\beta= \sum_{i \in \I_\beta} -b_i$ added.
\end{enumerate}\label{proposition:muPrelevant_cov2}\label{prop:muFromBNoRel_cov2}
\end{proposition}
 
\begin{proposition}
Let $X_A\subset \p^{n-1}$ be a projective toric variety with ${\rm codim}(X_A)=2$ and set $P={\rm Conv}(A)$. Let $\beta \subset \alpha$ be proper faces of $P$. Then the subdiagram volume $\mu'(\alpha,\beta)$ is as follows:\begin{enumerate}
\item[(a)] if not all rows of $B_\alpha$ are contained in a relevant line or if all rows of $B_\beta$ are contained in a relevant line then $$
\mu'(\alpha,\beta) =\frac{\left[\Z^{d}:\Z A_{\alpha}\right]}{\left[\Z^{d}:\Z A_{\beta}\right]} ,
$$
\item[(b)] if not all rows of $B_\beta$, but all rows of $B_\alpha$,  are contained in a relevant line, then $$
\mu'(\alpha,\beta) =\frac{\left[\Z^{d}:\Z A_{\alpha}\right]}{\left[\Z^{d}:\Z A_{\beta}\right]}\cdot \min\left( \sum_{b_i\;:\; \gamma_i>0}|\gamma_i|, \sum_{b_i\;:\; \gamma_i<0}|\gamma_i|\right),
$$
where $\gamma_i=\det(v,b_i)$ and $i$ loops over all $i \in \I_\beta$.
\end{enumerate}
\end{proposition}
 
\bibliography{ref}

\newcommand{\etalchar}[1]{$^{#1}$}
\begin{thebibliography}{DHO{\etalchar{+}}16}

\bibitem[AH17]{AH17}
Michael~F Adamer and Martin Helmer.
\newblock Euclidean distance degree for chemical reaction networks.
\newblock {\em arXiv preprint arXiv:1707.07650}, 2017.

\bibitem[Alu18]{aluffis2016paper}
Paolo Aluffi.
\newblock Projective duality and a chern-mather involution.
\newblock {\em Transactions of the American Mathematical Society},
  370(3):1803--1822, 2018.

\bibitem[BEF00]{bueler2000exact}
Benno B{\"u}eler, Andreas Enge, and Komei Fukuda.
\newblock Exact volume computation for polytopes: a practical study.
\newblock In {\em Polytopes--combinatorics and computation}, pages 131--154.
  Springer, 2000.

\bibitem[BGH{\etalchar{+}}10]{bank2010geometry}
Bernd Bank, Marc Giusti, Joos Heintz, Mohab~Safey El~Din, and Eric Schost.
\newblock On the geometry of polar varieties.
\newblock {\em Applicable Algebra in Engineering, Communication and Computing},
  21(1):33--83, 2010.

\bibitem[Cha93]{chazelle1993optimal}
Bernard Chazelle.
\newblock An optimal convex hull algorithm in any fixed dimension.
\newblock {\em Discrete \& Computational Geometry}, 10(1):377--409, 1993.

\bibitem[CLS11]{CLS}
David~A. Cox, John~B. Little, and Henry~K. Schenck.
\newblock {\em Toric varieties}, volume 124 of {\em Graduate Studies in
  Mathematics}.
\newblock American Mathematical Society, Providence, RI, 2011.

\bibitem[DF88]{dyer1988complexity}
Martin~E. Dyer and Alan~M. Frieze.
\newblock On the complexity of computing the volume of a polyhedron.
\newblock {\em SIAM Journal on Computing}, 17(5):967--974, 1988.

\bibitem[DHO{\etalchar{+}}16]{DHOST}
Jan Draisma, Emil Horobe{\c{t}}, Giorgio Ottaviani, Bernd Sturmfels, and
  Rekha~R Thomas.
\newblock The euclidean distance degree of an algebraic variety.
\newblock {\em Foundations of Computational Mathematics}, 16(1):99--149, 2016.

\bibitem[DS02]{DS}
Alicia Dickenstein and Bernd Sturmfels.
\newblock Elimination theory in codimension 2.
\newblock {\em J. Symbolic Comput.}, 34(2):119--135, 2002.

\bibitem[Ful13]{fulton}
William Fulton.
\newblock {\em Intersection theory}, volume~2.
\newblock Springer Science \& Business Media, 2013.

\bibitem[GKZ94]{GKZ}
I.~M. Gel$\prime$fand, M.~M. Kapranov, and A.~V. Zelevinsky.
\newblock {\em Discriminants, resultants, and multidimensional determinants}.
\newblock Mathematics: Theory \& Applications. Birkh\"auser Boston, Inc.,
  Boston, MA, 1994.

\bibitem[GS17]{M2}
Daniel~R. Grayson and Michael~E. Stillman.
\newblock {\em Macaulay2, a software system for research in algebraic
  geometry}, 2017.

\bibitem[Hol88]{holme1988geometric}
Audun Holme.
\newblock The geometric and numerical properties of duality in projective
  algebraic geometry.
\newblock {\em Manuscripta mathematica}, 61(2):145--162, 1988.

\bibitem[HS16]{HS}
M.~{Helmer} and B.~{Sturmfels}.
\newblock {Nearest Points on Toric Varieties}.
\newblock {\em To appear in Mathematica Scandinavica. Arxiv: 1603.06544}, 2016.

\bibitem[K{\etalchar{+}}86]{kleiman1986tangency}
Steven~L Kleiman et~al.
\newblock Tangency and duality.
\newblock In {\em Proceedings of the 1984 Vancouver conference in algebraic
  geometry}, volume~6, pages 163--225, 1986.

\bibitem[Kea98]{Keating}
Michael~Edward Keating.
\newblock {\em A first course in module theory}.
\newblock World Scientific, 1998.

\bibitem[LRAR13]{los2013role}
Ferdinand~CO Los, Tara~M Randis, Raffi~V Aroian, and Adam~J Ratner.
\newblock Role of pore-forming toxins in bacterial infectious diseases.
\newblock {\em Microbiology and Molecular Biology Reviews}, 77(2):173--207,
  2013.

\bibitem[Mac74]{macpherson}
R.~D. MacPherson.
\newblock Chern classes for singular algebraic varieties.
\newblock {\em Ann. of Math. (2)}, 100:423--432, 1974.

\bibitem[MS04]{MScca}
Ezra Miller and Bernd Sturmfels.
\newblock {\em Combinatorial commutative algebra}, volume 227.
\newblock Springer Science \& Business Media, 2004.

\bibitem[MT11]{MT}
Yutaka Matsui and Kiyoshi Takeuchi.
\newblock A geometric degree formula for {$A$}-discriminants and {E}uler
  obstructions of toric varieties.
\newblock {\em Adv. Math.}, 226(2):2040--2064, 2011.

\bibitem[N{\o}d18]{BIUN}
Bernt Ivar~Utst{\o}l N{\o}dland.
\newblock Local euler obstructions of toric varieties.
\newblock {\em Journal of Pure and Applied Algebra}, 222(3):508--533, 2018.

\bibitem[OPVV14]{degreeLatticeIdeals}
Liam O'Carroll, Francesc Planas-Vilanova, and Rafael~H Villarreal.
\newblock Degree and algebraic properties of lattice and matrix ideals.
\newblock {\em SIAM Journal on Discrete Mathematics}, 28(1):394--427, 2014.

\bibitem[OSS14]{ottaviani2014exact}
Giorgio Ottaviani, Pierre-Jean Spaenlehauer, and Bernd Sturmfels.
\newblock Exact solutions in structured low-rank approximation.
\newblock {\em SIAM Journal on Matrix Analysis and Applications},
  35(4):1521--1542, 2014.

\bibitem[Pie78]{piene1978polar}
Ragni Piene.
\newblock Polar classes of singular varieties.
\newblock {\em Ann. Sci. {\'E}cole Norm. Sup.(4)}, 11(2):247--276, 1978.

\bibitem[Pie15]{PienePolar}
Ragni Piene.
\newblock Polar varieties revisited.
\newblock In {\em Computer algebra and polynomials}, volume 8942 of {\em
  Lecture Notes in Comput. Sci.}, pages 139--150. Springer, Cham, 2015.

\bibitem[{Pie}16]{PieneCM}
R.~{Piene}.
\newblock {Chern-Mather classes of toric varieties}.
\newblock {\em ArXiv e-prints}, April 2016.

\bibitem[Rot10]{AMA}
Joseph~J Rotman.
\newblock {\em Advanced modern algebra}, volume 114.
\newblock American Mathematical Soc., 2010.

\bibitem[Tev06]{tevelev2006projective}
Evgueni~A Tevelev.
\newblock {\em Projective duality and homogeneous spaces}, volume 133.
\newblock Springer Science \& Business Media, 2006.

\end{thebibliography}
\bibliographystyle{alpha}
\end{document}